\documentclass[preprint,3p,11pt,authoryear]{elsarticle}

\makeatletter
\def\ps@pprintTitle{
 \let\@oddhead\@empty
 \let\@evenhead\@empty
 \def\@oddfoot{\centerline{\thepage}}
 \let\@evenfoot\@oddfoot}
\makeatother

\usepackage{amsthm,amsmath,amsfonts,amssymb,amsthm,graphicx,url}
\usepackage[colorlinks]{hyperref}
\usepackage[labelfont=bf]{caption}

\usepackage{mathtools}
\mathtoolsset{showonlyrefs=true} 
\usepackage{dsfont} 
\usepackage{mathrsfs} 
\usepackage{subcaption}
\usepackage{float}
\usepackage{geometry}
\theoremstyle{plain}
\newtheorem{theorem}{Theorem}
\newtheorem{lemma}{Lemma}
\newtheorem*{notation}{Notation}

\newcommand{\N}{\mathbb{N}}
\newcommand{\R}{\mathbb{R}}

\newcommand{\QQ}{\mathbb{Q}}
\newcommand{\PP}{\mathbb{P}}
\newcommand{\EE}{\mathbb{E}}

\newcommand{\VV}{\mathbb{V}\mathrm{ar}}

\newcommand{\nvert}[0]{\, \vert \, }
\newcommand{\bb}[1]{\boldsymbol{#1}}

\newcommand{\mat}[1]{\mathbb{#1}}

\newcommand{\OO}{\mathcal O}
\newcommand{\oo}{\mathrm{o}}
\newcommand{\leqdef}{\vcentcolon=}

\newcommand{\rd}{{\rm d}}
\newcommand{\ind}{\mathds{1}}

\allowdisplaybreaks

\begin{document}

\begin{frontmatter}

    \title{Local normal approximations and probability metric bounds for the matrix-variate \texorpdfstring{$T$}{T} distribution and its application to Hotelling's \texorpdfstring{$T$}{T} statistic}%

    \author[a1,a2]{Fr\'ed\'eric Ouimet\texorpdfstring{\fnref{fn1}}{)}}%

    \address[a1]{McGill University, Montreal, QC H3A 0B9, Canada.}%
    \address[a2]{California Institute of Technology, Pasadena, CA 91125, USA.}%

    \ead{frederic.ouimet2@mcgill.ca}%

    \fntext[fn1]{F.\ O.\ is supported by a postdoctoral fellowship from the FRQNT (B3XR).}%

    \begin{abstract}
        In this paper, we develop local expansions for the ratio of the centered matrix-variate $T$ density to the centered matrix-variate normal density with the same covariances. The approximations are used to derive upper bounds on several probability metrics (such as the total variation and Hellinger distance) between the corresponding induced measures.
        This work extends some of the results of \cite{doi:10.1109/LCOMM.2015.2442576} and \cite{Ouimet_2022_Student_JCA} for the univariate Student distribution to the matrix-variate setting.
    \end{abstract}


\end{frontmatter}

\section{Introduction}\label{sec:intro}

    For any $n\in \N$, define the space of (real symmetric) positive definite matrices of size $n\times n$ as follows:
    \begin{equation}\label{eq:def:positive.definite.matrices}
        \mathcal{S}_{++}^{\hspace{0.3mm}n} \leqdef \left\{\mat{M}\in \R^{n\times n} : \text{$\mat{M}$ is symmetric and positive definite}\right\}.
    \end{equation}
    For $d, m\in \N$, $\nu > 0$, $M\in \R^{d \times m}$, $\Sigma\in \mathcal{S}_{++}^{\hspace{0.3mm}d}$ and $\Omega\in \mathcal{S}_{++}^{\hspace{0.3mm}m}$, the density function of the centered (and normalized) matrix-variate $T$ distribution, hereafter denoted by $T_{d,m}(\nu,\Sigma,\Omega)$, is defined, for all $X\in \R^{d\times m}$, by
    \begin{equation}\label{eq:matrix.T.density}
        \begin{aligned}
            K_{\nu,\Sigma,\Omega}(X)
            &\leqdef \frac{\Gamma_d(\frac{1}{2}(\nu + m + d - 1))}{\Gamma_d(\frac{1}{2}(\nu + d - 1))} \frac{|\mathrm{I}_d + \nu^{-1} \Sigma^{-1} X \Omega^{-1} X^{\top}|^{-(\nu + m + d - 1)/2}}{(\nu\pi)^{md/2} |\Sigma|^{m/2} |\Omega|^{d/2}},
        \end{aligned}
    \end{equation}
    (see, e.g., (Definition~4.2.1 in \cite{Gupta_Nagar_1999}))
    where $\nu$ is the number of degrees of freedom, and
    \begin{equation}
        \begin{aligned}
            \Gamma_d(z)
            &= \int_{\mat{S}\in \mathcal{S}_{++}^{\hspace{0.3mm}d}} |\mat{S}|^{z-(d+1)/2} \exp(-\mathrm{tr}(\mat{S})) \rd \mat{S} \\
            &= \pi^{d(d-1)/4} \prod_{j=1}^d \Gamma\biggl(z-\frac{j-1}{2}\biggr), \quad \Re(z)>\frac{d-1}{2},
        \end{aligned}
    \end{equation}
    denotes the multivariate gamma function---see, e.g., (Section~35.3 in \cite{MR2723248}) and~\cite{MR3189307}---and
    \begin{equation}
        \Gamma(z) = \int_0^{\infty} t^{z - 1} e^{-t} \rd t, \quad \Re(z) > 0,
    \end{equation}
    is the classical gamma function.
    The mean and covariance matrix for the vectorization of $T\sim T_{d,m}(\nu,\Sigma,\Omega)$, namely
    \begin{equation}\label{eq:vectorization}
        \mathrm{vec}(T) \leqdef (T_{11}, T_{21}, \dots, T_{d1}, T_{12}, T_{22}, \dots, T_{d2}, \dots, T_{1m}, T_{2m}, \dots, T_{dm})^{\top},
    \end{equation}
    ($\mathrm{vec}(\cdot)$ is the operator that stacks the columns of a matrix on top of each other) are known to be (see, e.g., Theorem 4.3.1 in \cite{Gupta_Nagar_1999}, but be careful of the normalization):
    \begin{equation}
        \EE[\mathrm{vec}(T)] = \bb{0}_{dm} ~~(\text{i.e., } \EE[T] = 0_{d\times m}),
    \end{equation}
    and
    \begin{equation}\label{eq:variance.explicit.estimate}
        \VV(\mathrm{vec}(T^{\top})) = \frac{\nu}{(\nu - 2)} \Sigma \otimes \Omega, \quad \nu > 2.
    \end{equation}

    The first goal of our paper (Theorem~\ref{thm:LLT.matrix.T}) is to establish an asymptotic expansion for the ratio of the centered matrix-variate $T$ density \eqref{eq:matrix.T.density} to the centered matrix-variate normal (MN) density with the same covariances.
    According to (\citet{Gupta_Nagar_1999}, Theorem~2.2.1), the density of the $\mathrm{MN}_{d,m}(0_{d\times m}, \Sigma \otimes \Omega)$ distribution is
    \begin{equation}\label{eq:sym.matrix.normal.density}
        g_{\hspace{0.3mm}\Sigma,\Omega}(X) = \frac{\exp\left(-\frac{1}{2}\mathrm{tr}\left(\Sigma^{-1} X \Omega^{-1} X^{\top}\right)\right)}{(2\pi)^{md/2} |\Sigma|^{m/2} |\Omega|^{d/2}}, \quad X\in \R^{d\times m}.
    \end{equation}

    The second goal of our paper (Theorem~\ref{thm:probability.metric.bounds}) is to apply the log-ratio expansion from Theorem~\ref{thm:LLT.matrix.T} to derive upper bounds on multiple probability metrics between the measures induced by the centered matrix-variate $T$ distribution and the corresponding centered matrix-variate normal distribution. In the special case $m = 1$, this gives us probability metric upper bounds between the measure induced by Hotelling's $T$ statistic and the associated matrix-normal~measure.

    To give some practical motivations for the MN distribution \eqref{eq:sym.matrix.normal.density}, note that noise in the estimate of individual voxels of diffusion tensor magnetic resonance imaging (DT-MRI) data has been shown to be well modeled by a symmetric form of the $\mathrm{MN}_{3\times 3}$ distribution in~\cite{Pajevic_Basser_1999,doi:10.1002nbm.783,doi:10.1016/s1090-7807(02)00178-7}. The symmetric MN voxel distributions were combined into a tensor-variate normal distribution in \cite{doi:10.1109/TMI.2003.815059,doi:10.1137/16M1098693}, which could help to predict how the whole image (not just individual voxels) changes when shearing and dilation operations are applied in image wearing and registration problems; see \citet{doi:10.1109/42.963816}. In \cite{MR2485016}, maximum likelihood estimators and likelihood ratio tests are developed for the eigenvalues and eigenvectors of a form of the symmetric MN distribution with an orthogonally invariant covariance structure, both in one-sample problems (for example, in image interpolation) and two-sample problems (when comparing images) and under a broad variety of assumptions.
    This work extended significantly the previous results of \citet{MR131312}.
    In \cite{MR2485016}, it is also mentioned that the polarization pattern of cosmic microwave background (CMB) radiation measurements can be represented by $2\times 2$ positive definite matrices; see the primer by~\citet{doi:10.1016/S1384-1076(97)00022-5}. In a very recent and interesting paper, \citet{doi:10.1093/mnras/stab368} presented evidence for the Gaussianity of the local extrema of CMB maps.
    We can also mention \cite{doi:10.1016/j.patcog.2018.02.025}, where finite mixtures of skewed MN distributions were applied to an image recognition~problem.

    In general, we know that the Gaussian distribution is an attractor for sums of i.i.d.\ random variables with finite variance, which makes many estimators in statistics asymptotically normal. Similarly, we expect the MN distribution \eqref{eq:sym.matrix.normal.density} to be an attractor for sums of i.i.d.\ random matrices with finite variances (Hotelling's $T$-squared statistic is the most natural example), thus including many estimators, such as sample covariance matrices and score statistics for matrix parameters.
    In particular, if a given statistic or estimator is a function of the components of a sample covariance matrix for i.i.d.\ observations coming from a multivariate Gaussian population, then we could study its large sample properties (such as its moments) using Theorem~\ref{thm:LLT.matrix.T} (for example, by turning a Student-moments estimation problem into a Gaussian-moments estimation problem).

    The following is a brief outline of the paper.
    Our main results are stated in Section~\ref{sec:main.results} and proven in Section~\ref{sec:proofs}.
    Technical moment calculations are gathered in \ref{sec:technical.computations}.

    \begin{notation}
        {
        Throughout the paper, $a = \OO(b)$ means that $\limsup |a / b| < C$ as $\nu\to \infty$, where $C > 0$ is a universal constant.
        Whenever $C$ might depend on some parameter, we add a subscript (for example, $a = \OO_d(b)$).
        Similarly, $a = \oo(b)$ means that $\lim |a / b| =~0$, and subscripts indicate which parameters the convergence rate can depend on.
        If $a = (1 + \oo(1)) b$, then we write $a \sim b$.
        The notation $\mathrm{tr}(\cdot)$ will denote the trace operator for matrices and $|\cdot|$ their determinant.
        For a matrix $\mat{M}\in \R^{d\times d}$ that is diagonalizable, $\lambda_1(\mat{M}) \geq \dots \geq \lambda_d(\mat{M})$ will denote its eigenvalues, and we let $\bb{\lambda}(\mat{M}) \leqdef (\lambda_1(\mat{M}), \dots, \lambda_d(\mat{M}))^{\top}$.
        }
    \end{notation}

\section{Main Results}\label{sec:main.results}

    In Theorem~\ref{thm:LLT.matrix.T} below, we prove an asymptotic expansion for the ratio of the centered matrix-variate $T$ density to the centered matrix-variate normal (MN) density with the same covariances. The case $d = m = 1$ was proven recently in \cite{Ouimet_2022_Student_JCA} (see also \cite{doi:10.1109/LCOMM.2015.2442576} for an earlier rougher version).
    The result extends significantly the convergence in distribution result from Theorem~4.3.4 in \cite{Gupta_Nagar_1999}.

    \begin{theorem}\label{thm:LLT.matrix.T}
        Let $d,m\in \N$, $\Sigma\in \mathcal{S}_{++}^{\hspace{0.3mm}d}$ and $\Omega\in \mathcal{S}_{++}^{\hspace{0.3mm}m}$ be given.
        Pick any $\eta\in (0,1)$ and let
        \begin{equation}\label{eq:thm:p.k.expansion.condition}
            B_{\nu,\Sigma,\Omega}(\eta) \leqdef \left\{X\in \R^{d\times m} : \max_{1 \leq j \leq d} \frac{\delta_{\lambda_j}}{\sqrt{\nu - 2}} \leq \eta \, \nu^{-1/4}\right\}
        \end{equation}
        denote the bulk of the centered matrix-variate $T$ distribution, where
        \begin{equation}
            \Delta_X \leqdef \Sigma^{-1/2} X \Omega^{-1/2} \quad \text{and} \quad \delta_{\lambda_j} \leqdef \sqrt{\frac{\nu - 2}{\nu} \, \lambda_j(\Delta_X\Delta_X^{\top})}, \quad 1 \leq j \leq d.
        \end{equation}
        Then, as $\nu\to \infty$ and uniformly for $X\in B_{\nu,\Sigma,\Omega}(\eta)$, we have
        \begin{equation}\label{eq:LLT.order.2.log}
            \begin{aligned}
                &\log \left(\frac{[\nu / (\nu - 2)]^{md/2} \, K_{\nu,\Sigma,\Omega}(X)}{g_{\hspace{0.3mm}\Sigma,\Omega}(X / \sqrt{\nu / (\nu - 2)})}\right) \\
                &\hspace{20mm}= \nu^{-1} \left\{\hspace{-1mm}
                    \begin{array}{l}
                        \frac{1}{4} \mathrm{tr}\left((\Delta_X \Delta_X^{\top})^2\right) - \frac{(m + d + 1)}{2} \mathrm{tr}\left(\Delta_X \Delta_X^{\top}\right) \\
                        + \frac{m d (m + d + 1)}{4}
                    \end{array}
                    \hspace{-1mm}\right\} \\
                &\hspace{20mm}+ \nu^{-2} \left\{\hspace{-1mm}
                    \begin{array}{l}
                        -\frac{1}{6} \mathrm{tr}\left((\Delta_X \Delta_X^{\top})^3\right) + \frac{(m + d - 1)}{4} \mathrm{tr}\left((\Delta_X \Delta_X^{\top})^2\right) \\[1mm]
                        + \frac{m d}{24} (13 - 2 d^2 - 3 d (-3 + m) + 9 m - 2 m^2)
                    \end{array}
                    \hspace{-1mm}\right\} \\
                &\hspace{20mm}+ \nu^{-3} \left\{\hspace{-1mm}
                    \begin{array}{l}
                        \frac{1}{8} \mathrm{tr}\left((\Delta_X \Delta_X^{\top})^4\right) - \frac{(m + d - 1)}{6} \mathrm{tr}\left((\Delta_X \Delta_X^{\top})^3\right) \\[1mm]
                        + \frac{m d}{24} \left(\hspace{-1mm}
                            \begin{array}{l}
                                26 + d^3 + 2 d^2 (-3 + m) + 11 m \\
                                - 6 m^2 + m^3 + d (11 - 9 m + 2 m^2)
                            \end{array}
                            \hspace{-1mm}\right)
                    \end{array}
                    \hspace{-1mm}\right\} \\
                &\hspace{20mm}+ \OO_{d,m,\eta}\left(\frac{1 + \mathrm{tr}\left((\Delta_X \Delta_X^{\top})^5\right)}{\nu^4}\right).
            \end{aligned}
        \end{equation}
    \end{theorem}

    {Local approximations }such as the one in Theorem~\ref{thm:LLT.matrix.T} can be found for the Poisson, binomial and negative binomial distributions in \cite{MR207011} (based on Fourier analysis results from \cite{MR14626}), and \cite{MR538319} for the binomial distribution.
    Another approach, using Stein's method, is used to study the variance-gamma distribution in \cite{MR3194737}. Moreover, Kolmogorov and Wasserstein distance bounds are derived in \cite{MR4291370,MR4064309} for the Laplace and variance-gamma distributions.

    Below, we provide numerical evidence (displayed graphically) for the validity of the expansion in Theorem~\ref{thm:LLT.matrix.T} when $d = m = 2$.
    We compare three levels of approximation for various choices of $\mat{S}$.
    For any given $\mat{S}\in \mathcal{S}_{++}^{\hspace{0.3mm}d}$, define
    \begin{align}
        E_0
        &\leqdef \sup_{X\in B_{\nu,\Sigma,\Omega}(\nu^{-1/4})} \left|\log \left(\frac{[\nu / (\nu - 2)]^{md/2} \, K_{\nu,\Sigma,\Omega}(X)}{g_{\hspace{0.3mm}\Sigma,\Omega}(X / \sqrt{\nu / (\nu - 2)})}\right)\right|, \label{eq:E.0} \\[2.5mm]
        E_1
        &\leqdef \sup_{X\in B_{\nu,\Sigma,\Omega}(\nu^{-1/4})} \left|\log \left(\frac{[\nu / (\nu - 2)]^{md/2} \, K_{\nu,\Sigma,\Omega}(X)}{g_{\hspace{0.3mm}\Sigma,\Omega}(X / \sqrt{\nu / (\nu - 2)})}\right)\right. \notag \\
        &\hspace{25mm}\left.- \nu^{-1} \left\{\frac{1}{4} \mathrm{tr}\left((\Delta_X \Delta_X^{\top})^2\right) - \frac{(m + d + 1)}{2} \mathrm{tr}\left(\Delta_X \Delta_X^{\top}\right) + \frac{m d (m + d + 1)}{4}\right\}\right|, \label{eq:E.1} \\[1.5mm]
        E_2
        &\leqdef \sup_{X\in B_{\nu,\Sigma,\Omega}(\nu^{-1/4})} \left|\log \left(\frac{[\nu / (\nu - 2)]^{md/2} \, K_{\nu,\Sigma,\Omega}(X)}{g_{\hspace{0.3mm}\Sigma,\Omega}(X / \sqrt{\nu / (\nu - 2)})}\right)\right. \notag \\
        &\hspace{25mm}\left.- \nu^{-1} \left\{\frac{1}{4} \mathrm{tr}\left((\Delta_X \Delta_X^{\top})^2\right) - \frac{(m + d + 1)}{2} \mathrm{tr}\left(\Delta_X \Delta_X^{\top}\right) + \frac{m d (m + d + 1)}{4}\right\}\right. \notag \\
        &\hspace{25mm}\left.- \nu^{-2} \left\{\hspace{-1mm}
        \begin{array}{l}
            -\frac{1}{6} \mathrm{tr}\left((\Delta_X \Delta_X^{\top})^3\right) + \frac{(m + d - 1)}{4} \mathrm{tr}\left((\Delta_X \Delta_X^{\top})^2\right) \\[1mm]
            + \frac{m d}{24} (13 - 2 d^2 - 3 d (-3 + m) + 9 m - 2 m^2)
        \end{array}
        \hspace{-1mm}\right\}\right|. \label{eq:E.2}
    \end{align}
    In the \texttt{R} software \citep{Rsoftware}, we use Equation~\eqref{eq:LLT.beginning.next.1} to evaluate the log-ratios inside $E_0$, $E_1$ and $E_2$.

    Note that $X\in B_{\nu,\Sigma,\Omega}(\nu^{-1/4})$ implies $|\mathrm{tr}((\Delta_X \Delta_X^{\top})^k)| \leq d$ for all $k\in \N$, so we expect from Theorem~\ref{thm:LLT.matrix.T} that the maximum errors above ($E_0$, $E_1$ and $E_2$) will have the asymptotic behavior
    \begin{align}
        E_i = \OO_d(\nu^{-(1 + i)}), \quad \text{for all } i\in \{0,1,2\},
    \end{align}
    or, equivalently,
    \begin{align}\label{eq:liminf.exponent.bound}
        \liminf_{\nu\to \infty} \frac{\log E_i}{\log (\nu^{-1})} \geq 1 + i, \quad \text{for all } i\in \{0,1,2\}.
    \end{align}
    The property \eqref{eq:liminf.exponent.bound} is verified in Figure~\ref{fig:error.exponents.plots}
    below, for $\Omega = \mathrm{I}_2$ and various choices of $\Sigma_{2\times 2}$.
    Similarly, the corresponding log-log plots of the errors as a function of $\nu$ are displayed in Figure~\ref{fig:loglog.errors.plots}.
    The simulations are limited to the range $5 \leq \nu \leq 1005$.
    The \texttt{R} code that generated Figures~\ref{fig:error.exponents.plots}~and~\ref{fig:loglog.errors.plots} can be found   at Supplementary Material.

    As a consequence of the previous theorem, we can derive asymptotic upper bounds on several probability metrics between the probability measures induced by the centered matrix-variate $T$ distribution \eqref{eq:matrix.T.density} and the corresponding centered matrix-variate normal distribution \eqref{eq:sym.matrix.normal.density}. The distance between Hotelling's $T$ statistic \citep{doi:10.1214/aoms/1177732979} and the corresponding matrix-variate normal distribution is obtained in the special case $m = 1$.

    \begin{theorem}[Probability metric upper bounds]\label{thm:probability.metric.bounds}
        Let $d,m\in \N$, $\Sigma\in \mathcal{S}_{++}^{\hspace{0.3mm}d}$ and $\Omega\in \mathcal{S}_{++}^{\hspace{0.3mm}m}$ be given.
        Assume that $X\sim T_{d,m}(\nu,\Sigma,\Omega)$, $Y\sim \mathrm{MN}_{d,m}(0_{d\times m}, \Sigma \otimes \Omega)$, and let $\PP_{\nu,\Sigma,\Omega}$ and $\QQ_{\Sigma,\Omega}$ be the laws of $X$ and $Y \sqrt{\nu / (\nu - 2)}$, respectively.
        Then, as $\nu\to \infty$,
        \begin{equation}
            \mathrm{dist}\hspace{0.3mm}(\PP_{\nu,\Sigma,\Omega},\QQ_{\Sigma,\Omega}) \leq \frac{C \hspace{0.3mm} (m d)^{3/2}}{\nu} \qquad \text{and} \qquad \mathcal{H}(\PP_{\nu,\Sigma,\Omega},\QQ_{\Sigma,\Omega}) \leq \sqrt{\frac{2C \hspace{0.3mm} (m d)^{3/2}}{\nu}},
        \end{equation}
        where $C > 0$ is a universal constant, $\mathcal{H}(\cdot,\cdot)$ denotes the Hellinger distance, and $\mathrm{dist}\hspace{0.3mm}(\cdot,\cdot)$ can be replaced by any of the following probability metrics: total variation, Kolmogorov (or uniform) metric, L\'evy metric, discrepancy metric, Prokhorov metric.
    \end{theorem}

    \begin{figure}[H]
        \captionsetup[subfigure]{labelformat=empty}
        \vspace{-0.5cm}
        \centering
        \begin{subfigure}[b]{0.22\textwidth}
            \centering
            \includegraphics[width=\textwidth, height=0.85\textwidth]{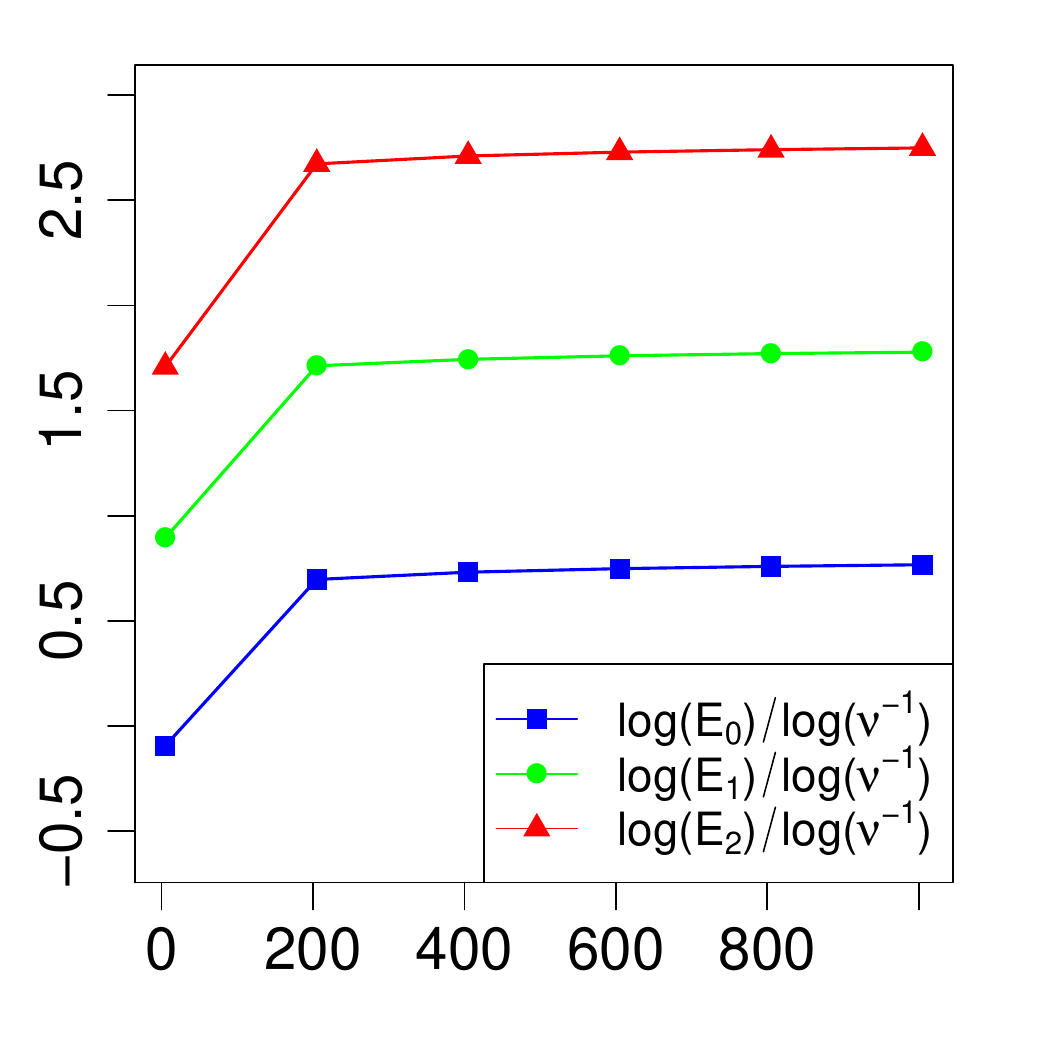}
            \vspace{-0.8cm}
            \caption{\scriptsize $\Sigma = \begin{pmatrix} 2 &\hspace{-2mm} 1 \\ 1 &\hspace{-2mm} 2\end{pmatrix}$, $\Omega = \begin{pmatrix} 1 &\hspace{-2mm} 0 \\ 0 &\hspace{-2mm} 1\end{pmatrix}$}
        \end{subfigure}
        ~~
        \begin{subfigure}[b]{0.22\textwidth}
            \centering
            \includegraphics[width=\textwidth, height=0.85\textwidth]{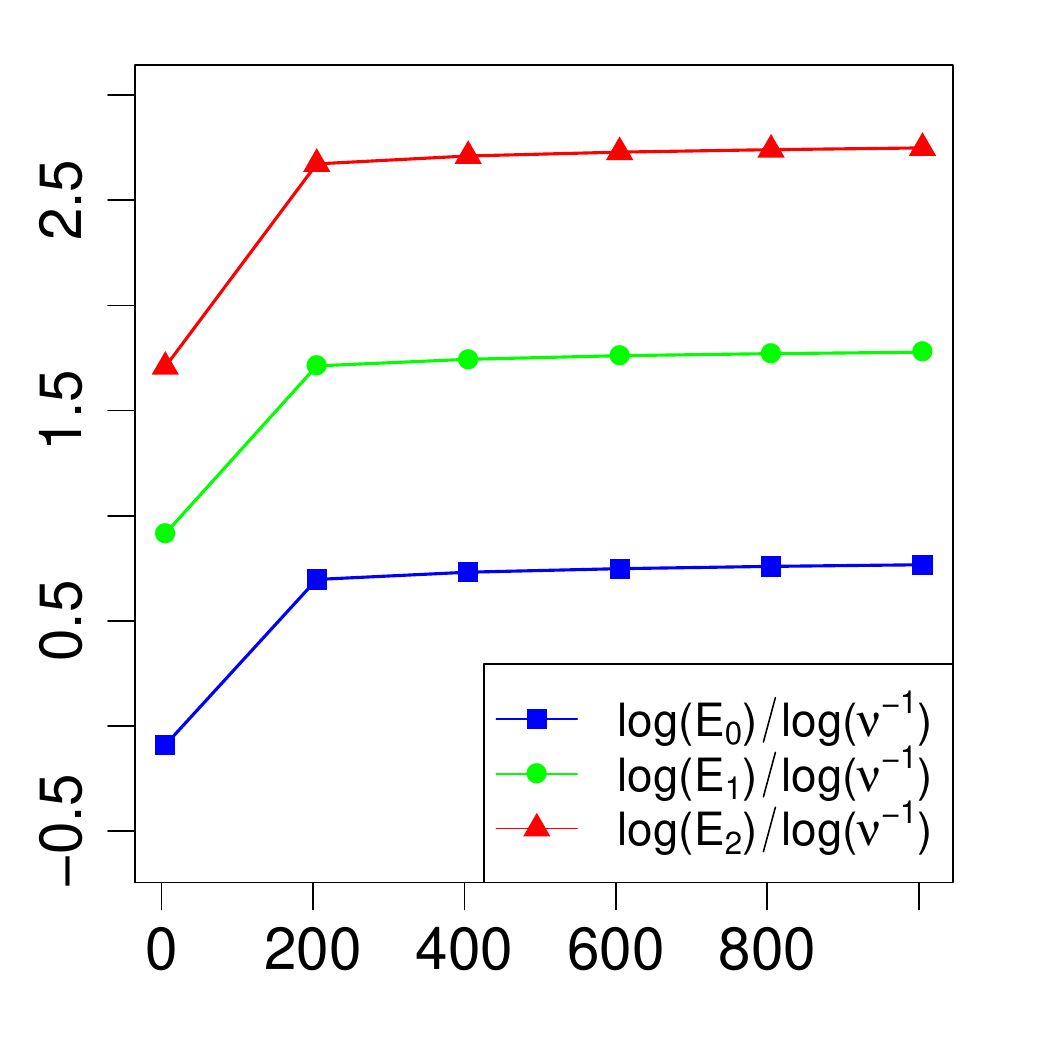}
            \vspace{-0.8cm}
            \caption{\scriptsize $\Sigma = \begin{pmatrix} 2 &\hspace{-2mm} 1 \\ 1 &\hspace{-2mm} 3\end{pmatrix}$, $\Omega = \begin{pmatrix} 1 &\hspace{-2mm} 0 \\ 0 &\hspace{-2mm} 1\end{pmatrix}$}
        \end{subfigure}
        ~~
        \begin{subfigure}[b]{0.22\textwidth}
            \centering
            \includegraphics[width=\textwidth, height=0.85\textwidth]{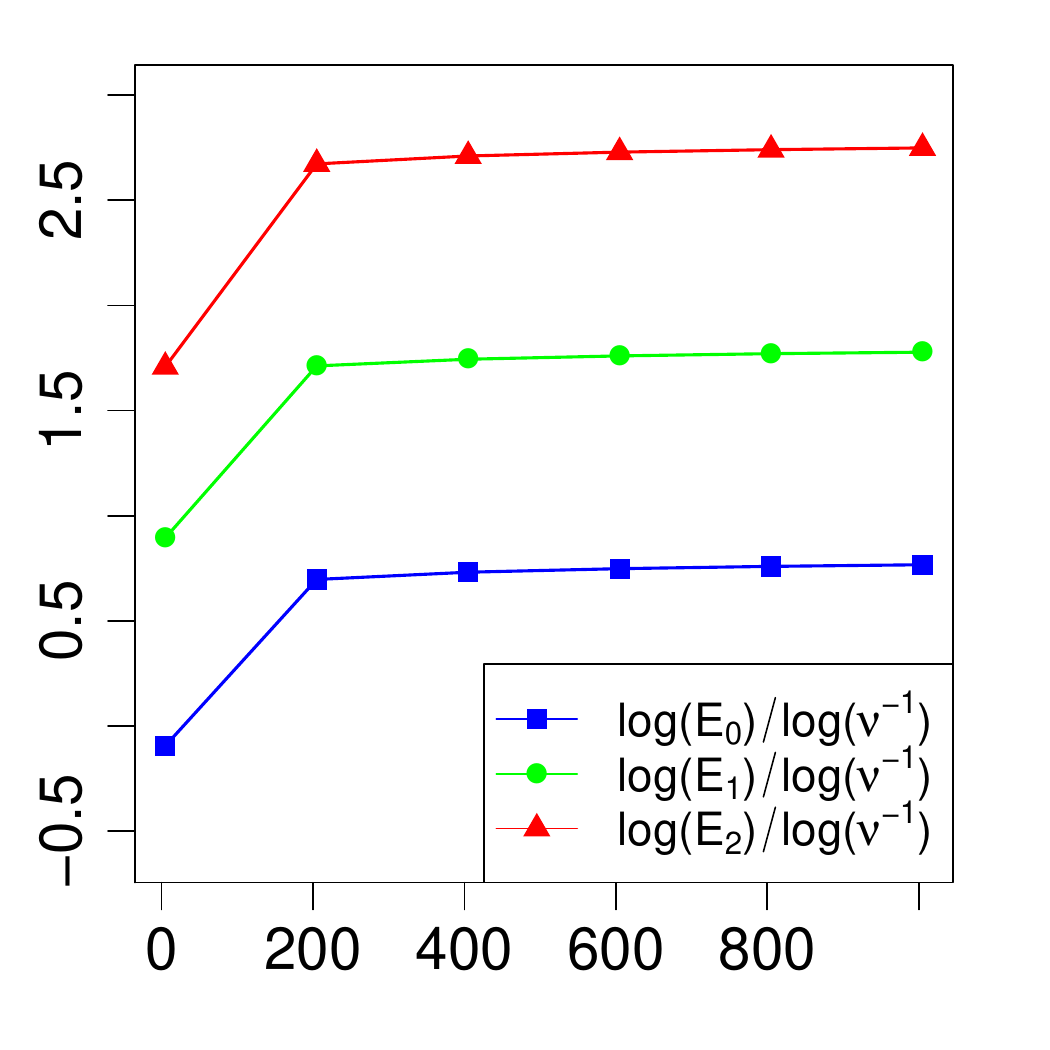}
            \vspace{-0.8cm}
            \caption{\scriptsize $\Sigma = \begin{pmatrix} 2 &\hspace{-2mm} 1 \\ 1 &\hspace{-2mm} 4\end{pmatrix}$, $\Omega = \begin{pmatrix} 1 &\hspace{-2mm} 0 \\ 0 &\hspace{-2mm} 1\end{pmatrix}$}
        \end{subfigure}
        ~~
        \begin{subfigure}[b]{0.22\textwidth}
            \centering
            \includegraphics[width=\textwidth, height=0.85\textwidth]{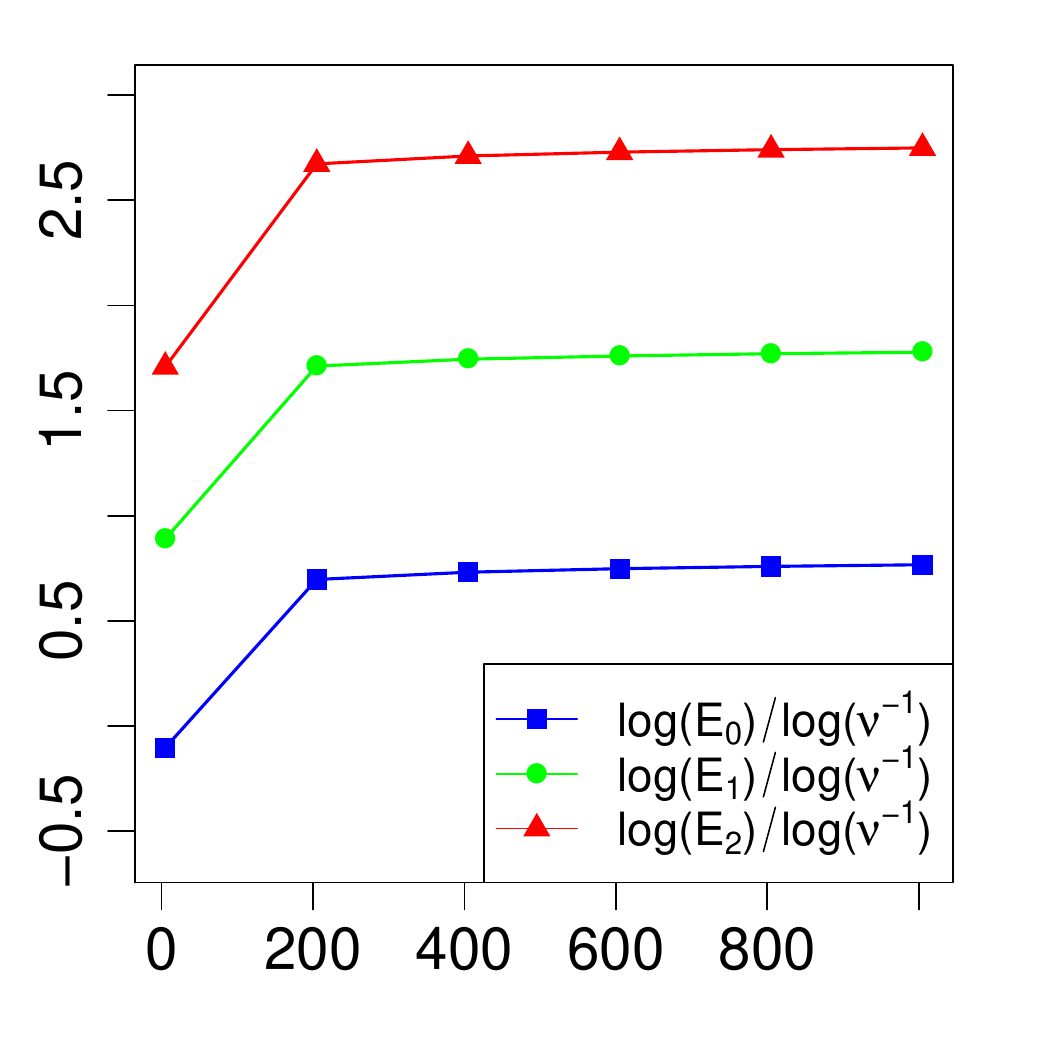}
            \vspace{-0.8cm}
            \caption{\scriptsize $\Sigma = \begin{pmatrix} 2 &\hspace{-2mm} 1 \\ 1 &\hspace{-2mm} 5\end{pmatrix}$, $\Omega = \begin{pmatrix} 1 &\hspace{-2mm} 0 \\ 0 &\hspace{-2mm} 1\end{pmatrix}$}
        \end{subfigure}
        \begin{subfigure}[b]{0.22\textwidth}
            \centering
            \includegraphics[width=\textwidth, height=0.85\textwidth]{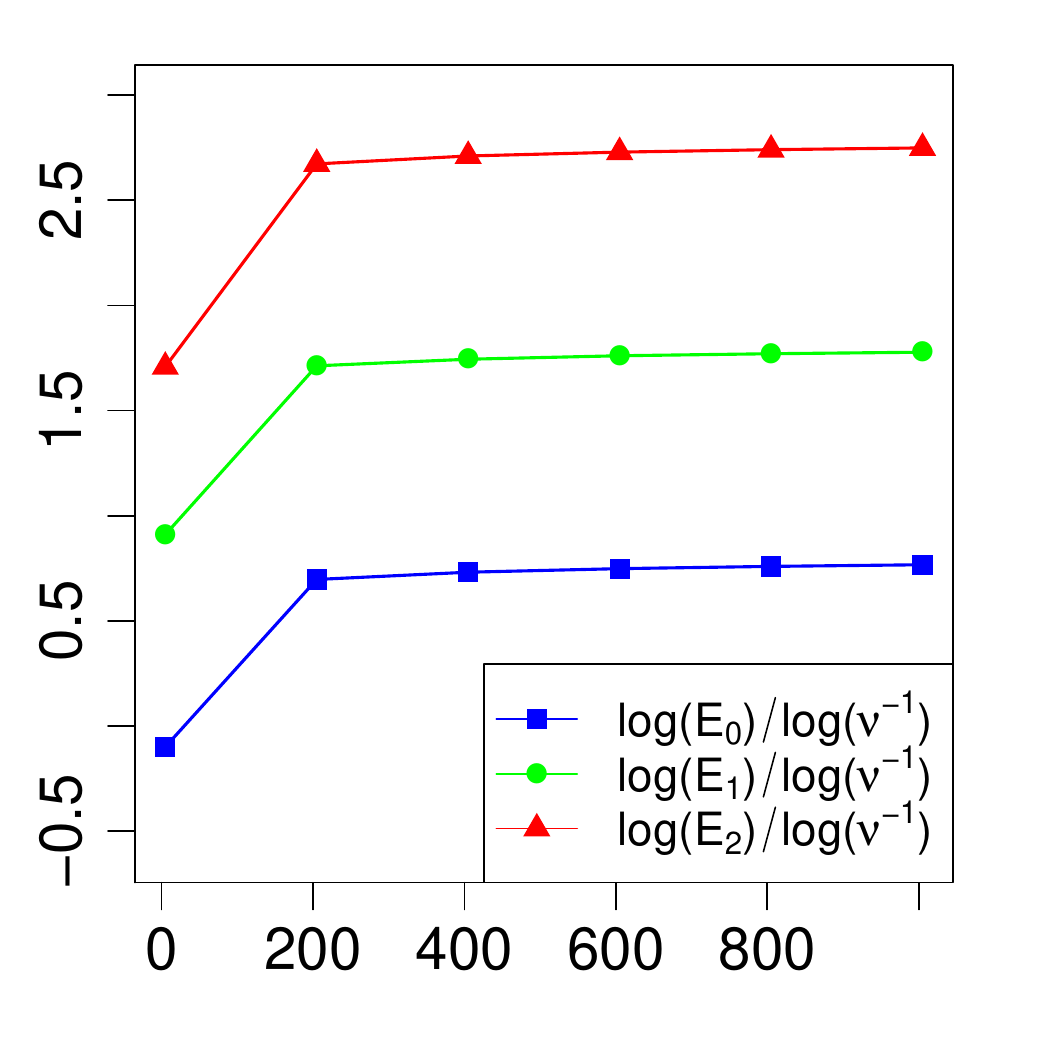}
            \vspace{-0.8cm}
            \caption{\scriptsize $\Sigma = \begin{pmatrix} 3 &\hspace{-2mm} 1 \\ 1 &\hspace{-2mm} 2\end{pmatrix}$, $\Omega = \begin{pmatrix} 1 &\hspace{-2mm} 0 \\ 0 &\hspace{-2mm} 1\end{pmatrix}$}
        \end{subfigure}
        ~~
        \begin{subfigure}[b]{0.22\textwidth}
            \centering
            \includegraphics[width=\textwidth, height=0.85\textwidth]{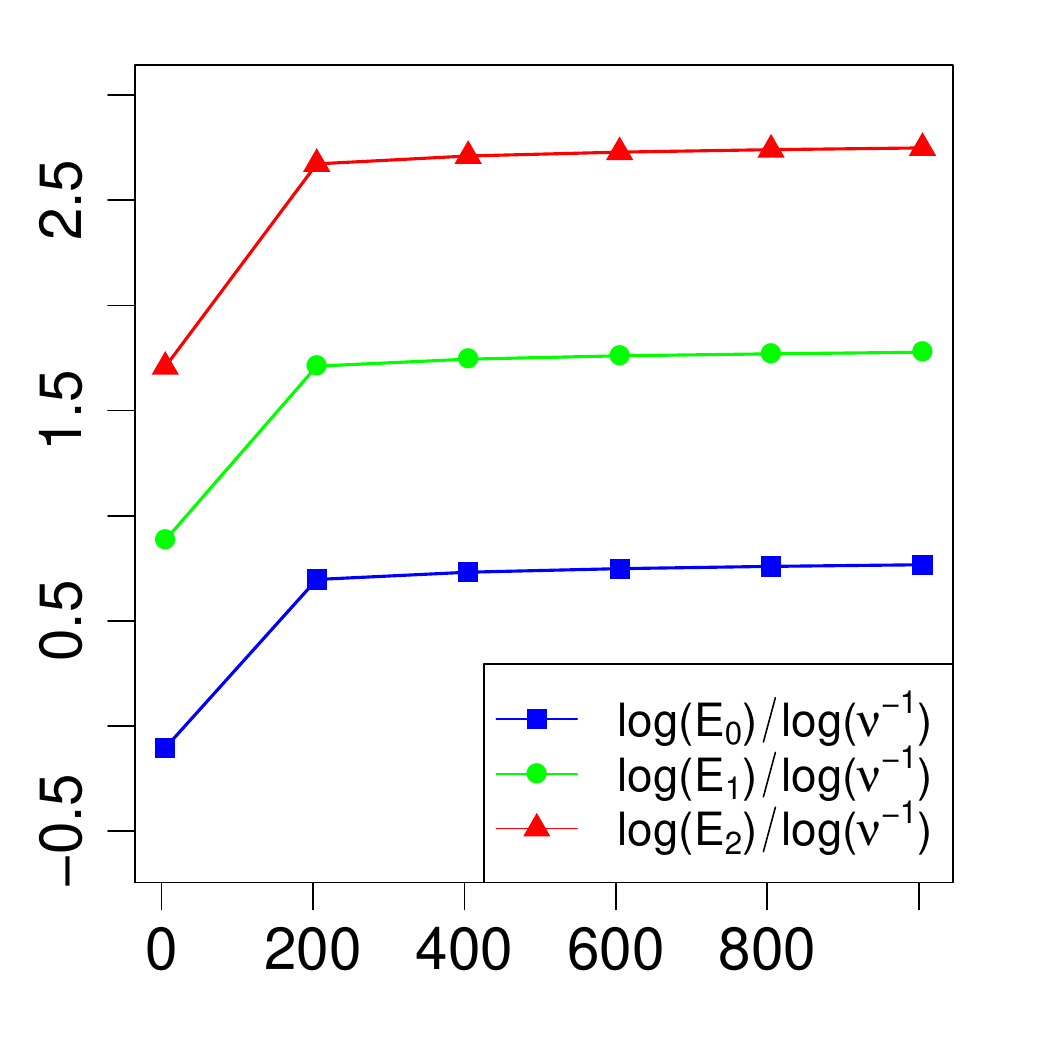}
            \vspace{-0.8cm}
            \caption{\scriptsize $\Sigma = \begin{pmatrix} 3 &\hspace{-2mm} 1 \\ 1 &\hspace{-2mm} 3\end{pmatrix}$, $\Omega = \begin{pmatrix} 1 &\hspace{-2mm} 0 \\ 0 &\hspace{-2mm} 1\end{pmatrix}$}
        \end{subfigure}
        ~~
        \begin{subfigure}[b]{0.22\textwidth}
            \centering
            \includegraphics[width=\textwidth, height=0.85\textwidth]{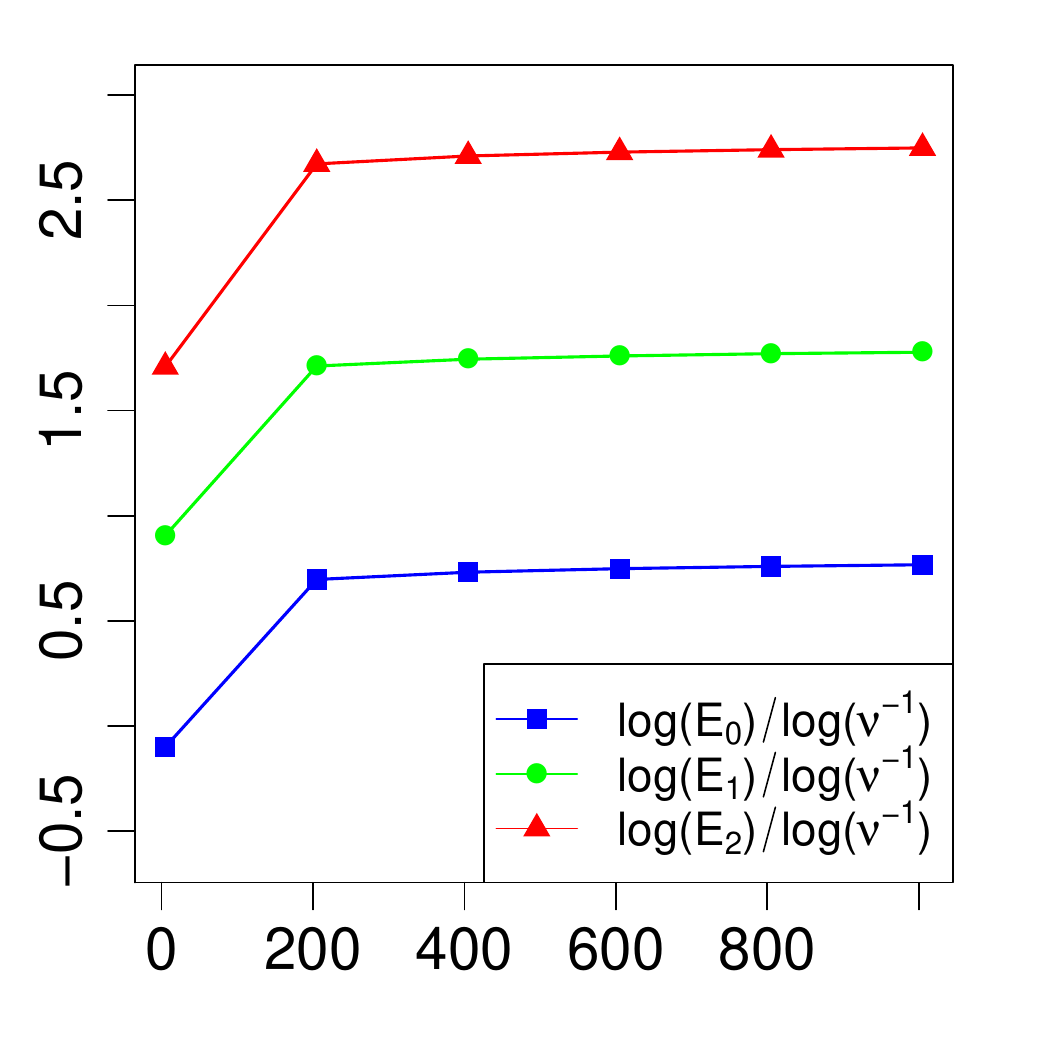}
            \vspace{-0.8cm}
            \caption{\scriptsize $\Sigma = \begin{pmatrix} 3 &\hspace{-2mm} 1 \\ 1 &\hspace{-2mm} 4\end{pmatrix}$, $\Omega = \begin{pmatrix} 1 &\hspace{-2mm} 0 \\ 0 &\hspace{-2mm} 1\end{pmatrix}$}
        \end{subfigure}
        ~~
        \begin{subfigure}[b]{0.22\textwidth}
            \centering
            \includegraphics[width=\textwidth, height=0.85\textwidth]{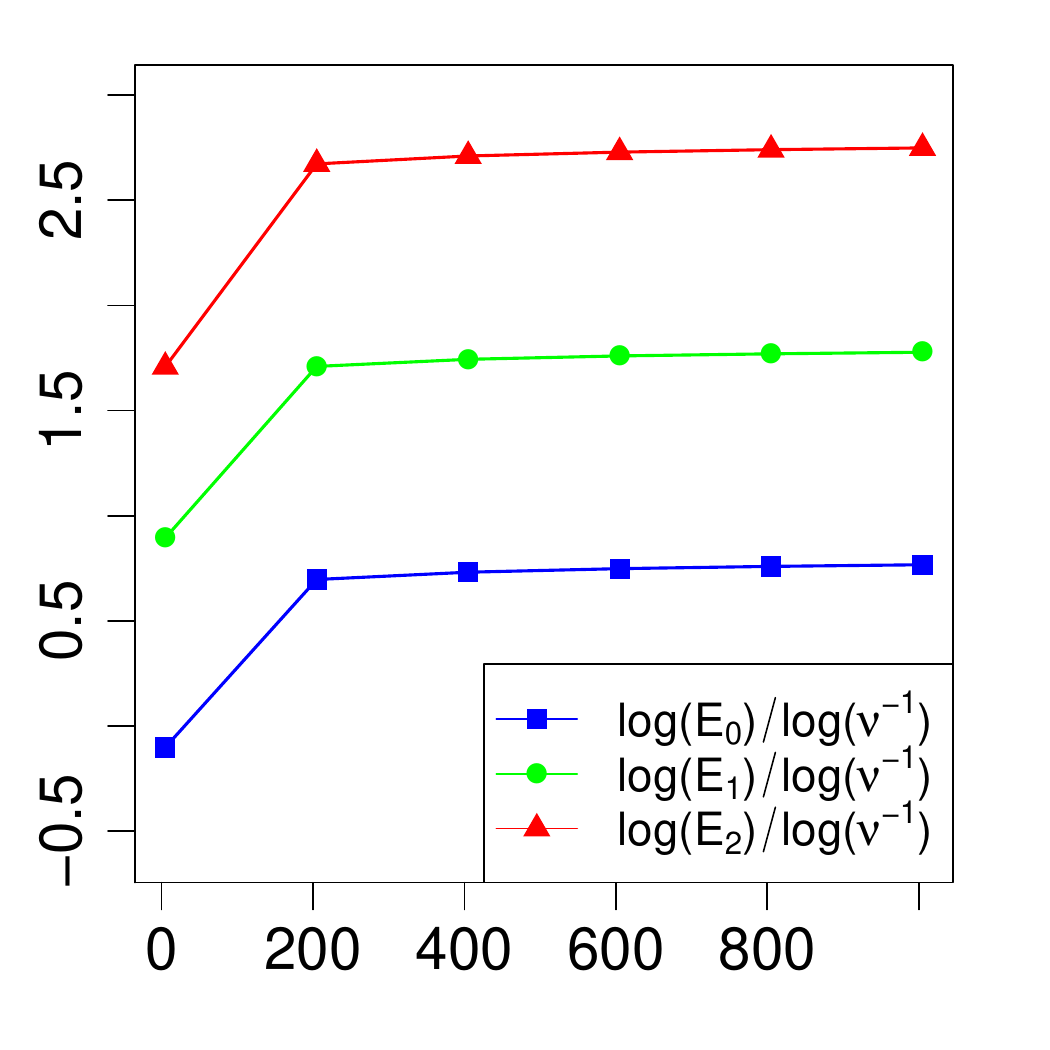}
            \vspace{-0.8cm}
            \caption{\scriptsize $\Sigma = \begin{pmatrix} 3 &\hspace{-2mm} 1 \\ 1 &\hspace{-2mm} 5\end{pmatrix}$, $\Omega = \begin{pmatrix} 1 &\hspace{-2mm} 0 \\ 0 &\hspace{-2mm} 1\end{pmatrix}$}
        \end{subfigure}
        \begin{subfigure}[b]{0.22\textwidth}
            \centering
            \includegraphics[width=\textwidth, height=0.85\textwidth]{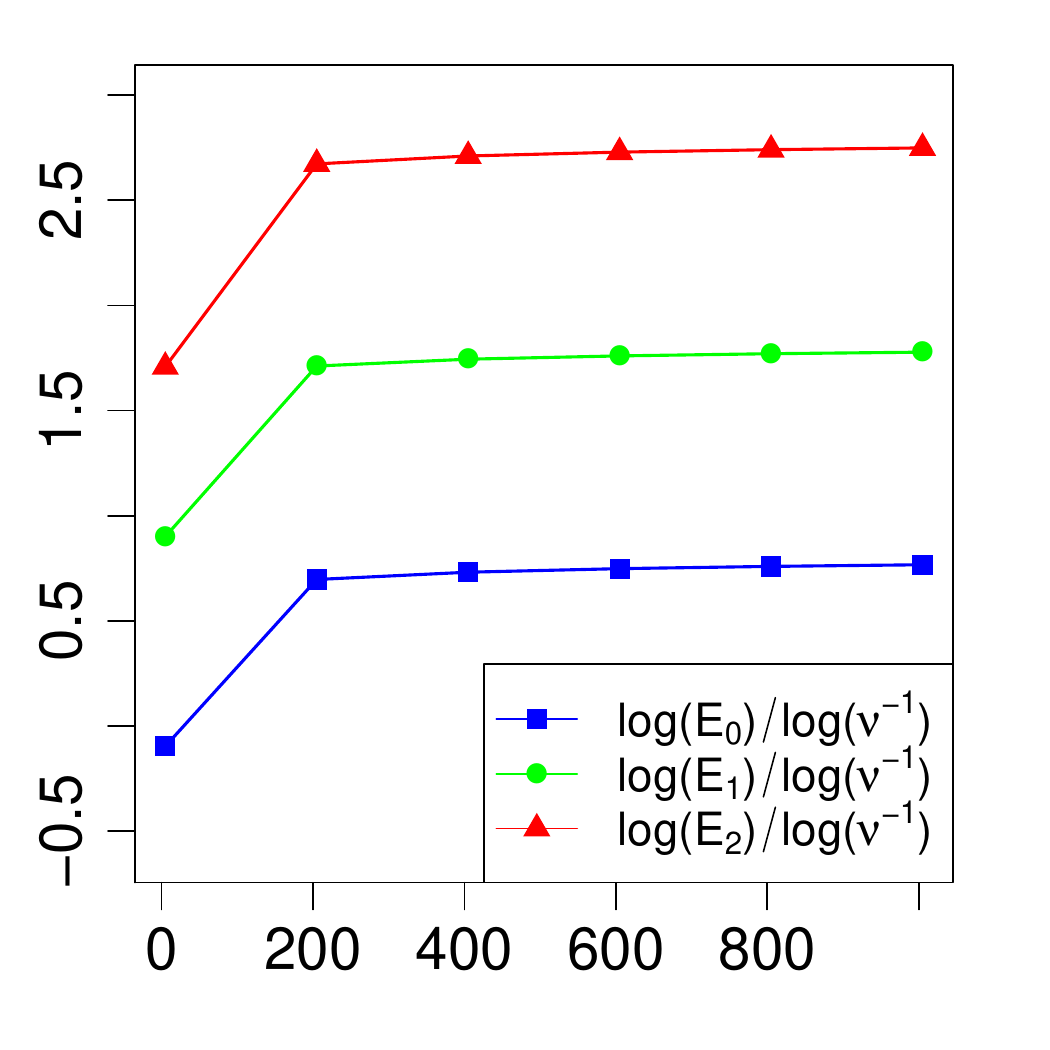}
            \vspace{-0.8cm}
            \caption{\scriptsize $\Sigma = \begin{pmatrix} 4 &\hspace{-2mm} 1 \\ 1 &\hspace{-2mm} 2\end{pmatrix}$, $\Omega = \begin{pmatrix} 1 &\hspace{-2mm} 0 \\ 0 &\hspace{-2mm} 1\end{pmatrix}$}
        \end{subfigure}
        ~~
        \begin{subfigure}[b]{0.22\textwidth}
            \centering
            \includegraphics[width=\textwidth, height=0.85\textwidth]{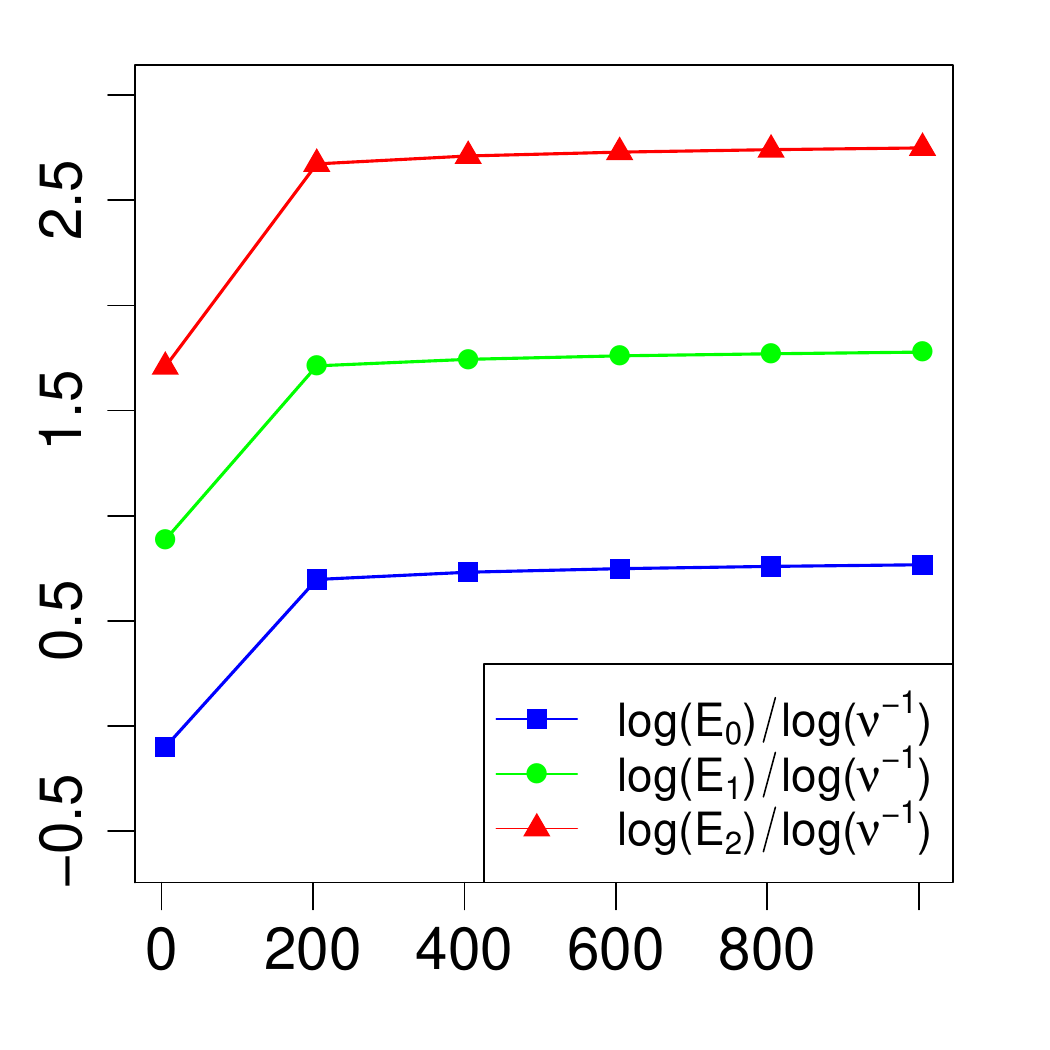}
            \vspace{-0.8cm}
            \caption{\scriptsize $\Sigma = \begin{pmatrix} 4 &\hspace{-2mm} 1 \\ 1 &\hspace{-2mm} 3\end{pmatrix}$, $\Omega = \begin{pmatrix} 1 &\hspace{-2mm} 0 \\ 0 &\hspace{-2mm} 1\end{pmatrix}$}
        \end{subfigure}
        ~~
        \begin{subfigure}[b]{0.22\textwidth}
            \centering
            \includegraphics[width=\textwidth, height=0.85\textwidth]{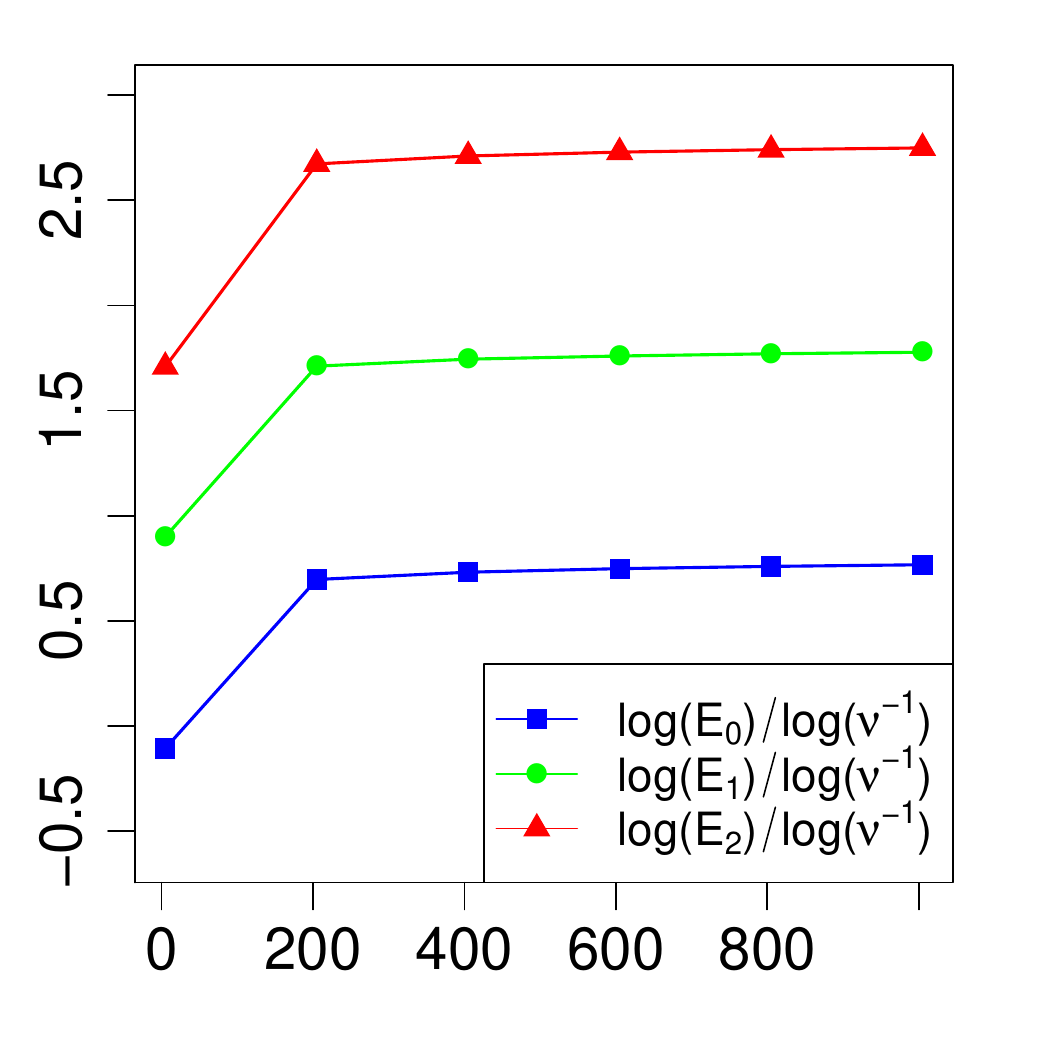}
            \vspace{-0.8cm}
            \caption{\scriptsize $\Sigma = \begin{pmatrix} 4 &\hspace{-2mm} 1 \\ 1 &\hspace{-2mm} 4\end{pmatrix}$, $\Omega = \begin{pmatrix} 1 &\hspace{-2mm} 0 \\ 0 &\hspace{-2mm} 1\end{pmatrix}$}
        \end{subfigure}
        ~~
        \begin{subfigure}[b]{0.22\textwidth}
            \centering
            \includegraphics[width=\textwidth, height=0.85\textwidth]{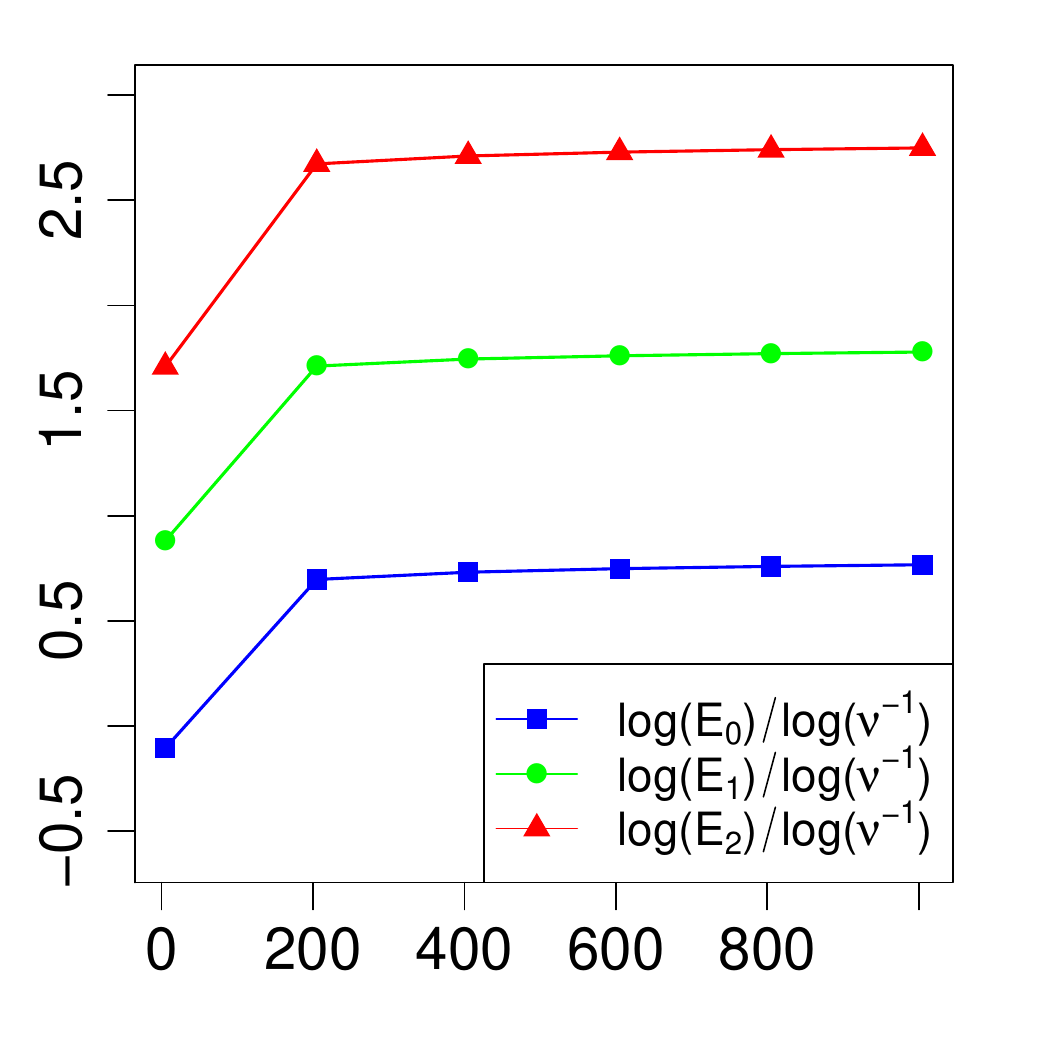}
            \vspace{-0.8cm}
            \caption{\scriptsize $\Sigma = \begin{pmatrix} 4 &\hspace{-2mm} 1 \\ 1 &\hspace{-2mm} 5\end{pmatrix}$, $\Omega = \begin{pmatrix} 1 &\hspace{-2mm} 0 \\ 0 &\hspace{-2mm} 1\end{pmatrix}$}
        \end{subfigure}
        \caption{\scriptsize Plots of $\log E_i / \log (\nu^{-1})$ as a function of $\nu$, for various choices of $\Sigma$. The plots confirm \eqref{eq:liminf.exponent.bound} for our choices of $\Sigma$ and bring strong evidence for the validity of Theorem~\ref{thm:LLT.matrix.T}.}
        \label{fig:error.exponents.plots}
    \end{figure}

    \vspace{3mm}
    \begin{figure}[H]
        \captionsetup[subfigure]{labelformat=empty}
        \vspace{-0.5cm}
        \centering
        \begin{subfigure}[b]{0.22\textwidth}
            \centering
            \includegraphics[width=\textwidth, height=0.85\textwidth]{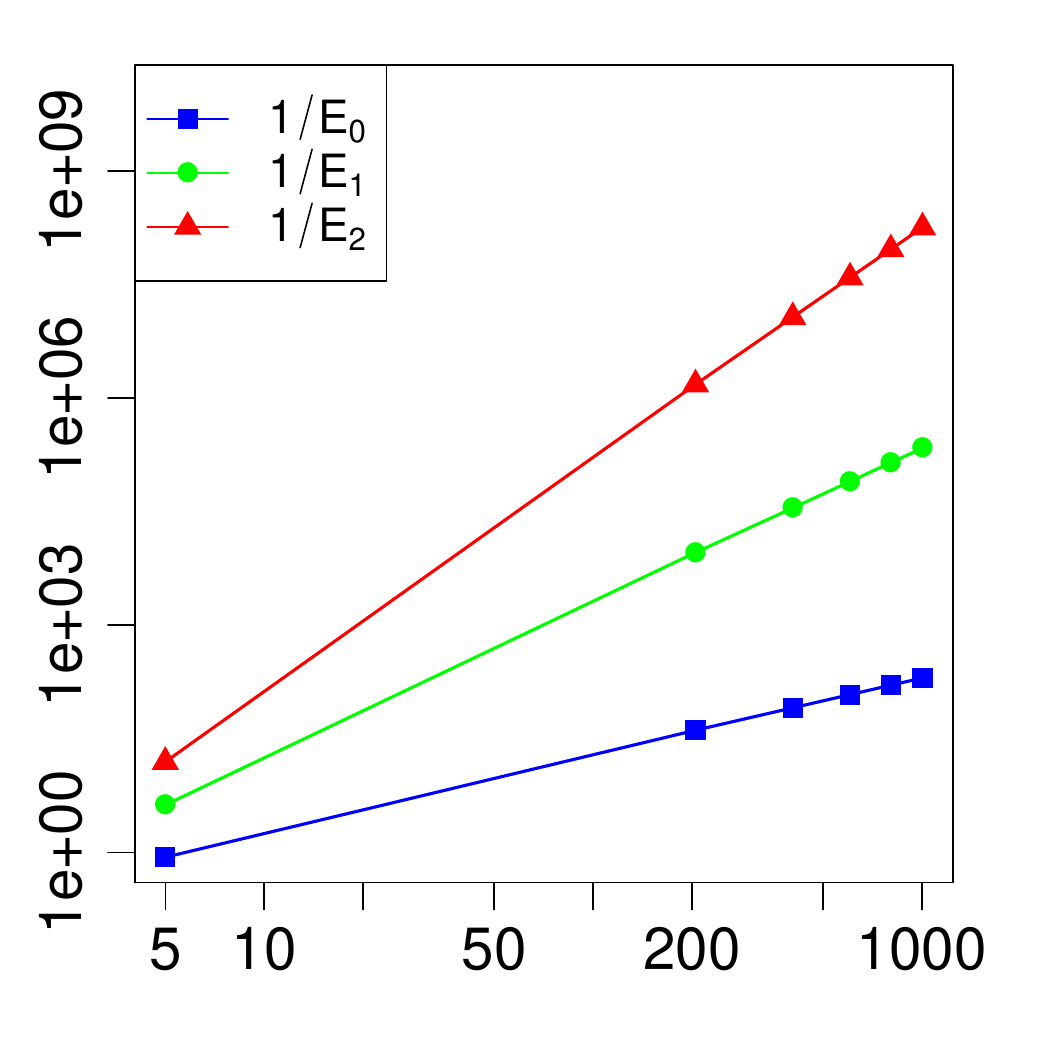}
            \vspace{-0.8cm}
            \caption{\scriptsize $\Sigma = \begin{pmatrix} 2 &\hspace{-2mm} 1 \\ 1 &\hspace{-2mm} 2\end{pmatrix}$, $\Omega = \begin{pmatrix} 1 &\hspace{-2mm} 0 \\ 0 &\hspace{-2mm} 1\end{pmatrix}$}
        \end{subfigure}
        ~~
        \begin{subfigure}[b]{0.22\textwidth}
            \centering
            \includegraphics[width=\textwidth, height=0.85\textwidth]{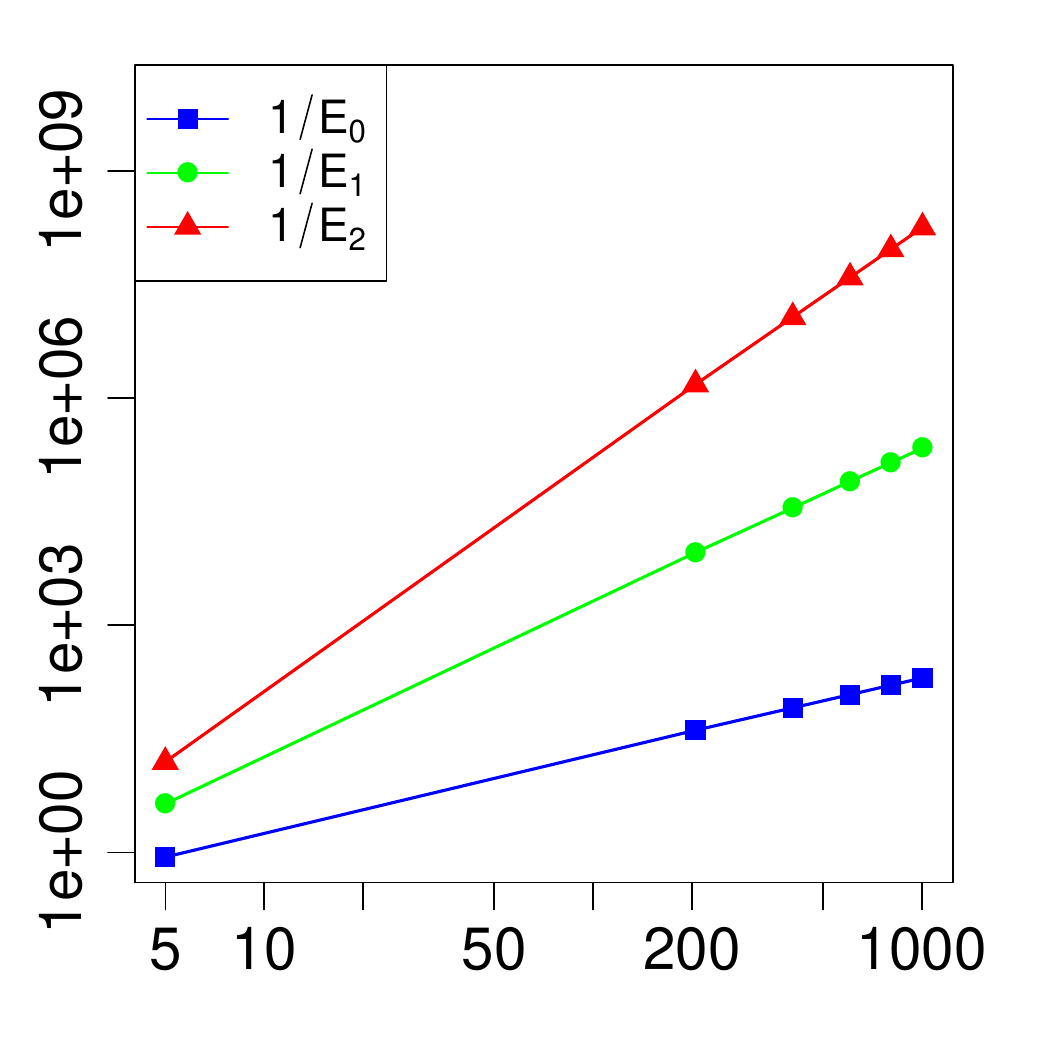}
            \vspace{-0.8cm}
            \caption{\scriptsize $\Sigma = \begin{pmatrix} 2 &\hspace{-2mm} 1 \\ 1 &\hspace{-2mm} 3\end{pmatrix}$, $\Omega = \begin{pmatrix} 1 &\hspace{-2mm} 0 \\ 0 &\hspace{-2mm} 1\end{pmatrix}$}
        \end{subfigure}
        ~~
        \begin{subfigure}[b]{0.22\textwidth}
            \centering
            \includegraphics[width=\textwidth, height=0.85\textwidth]{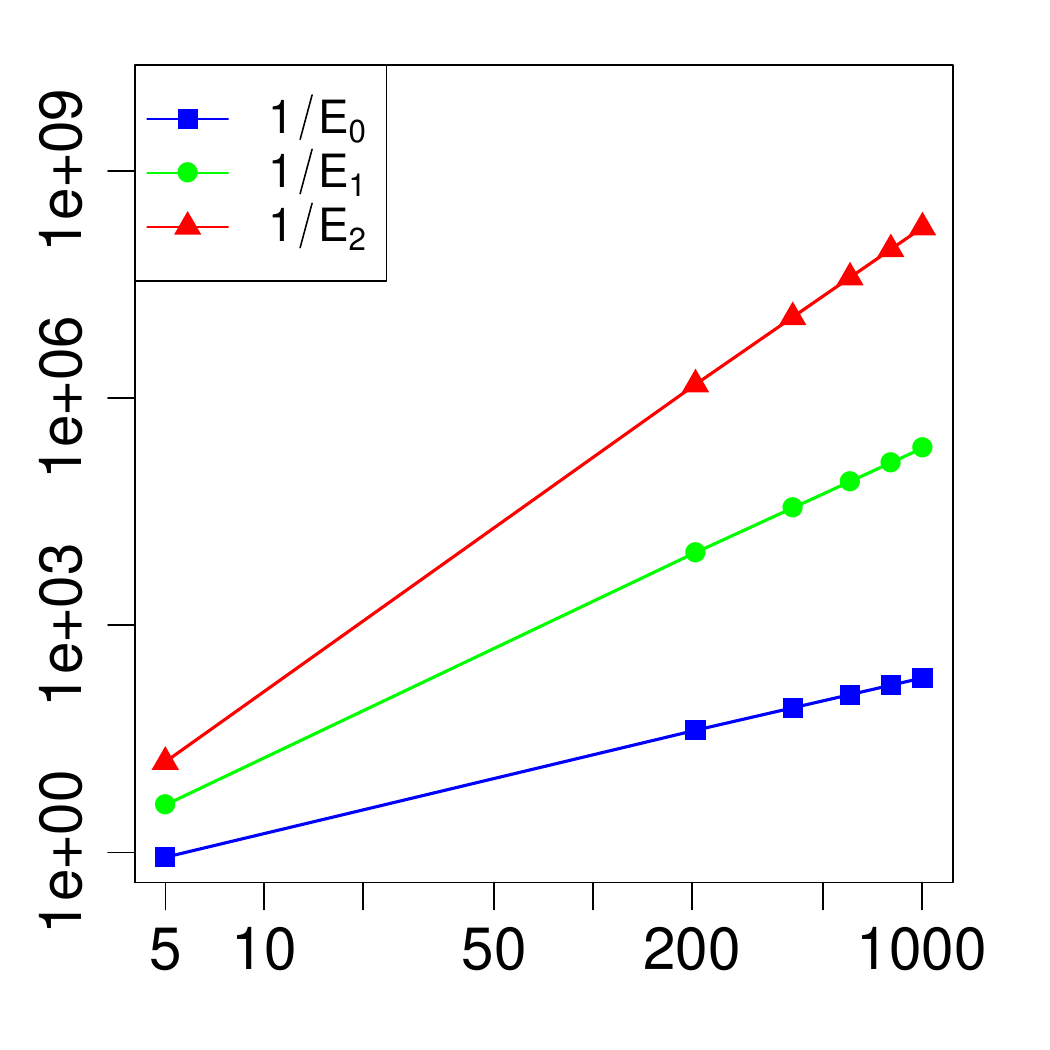}
            \vspace{-0.8cm}
            \caption{\scriptsize $\Sigma = \begin{pmatrix} 2 &\hspace{-2mm} 1 \\ 1 &\hspace{-2mm} 4\end{pmatrix}$, $\Omega = \begin{pmatrix} 1 &\hspace{-2mm} 0 \\ 0 &\hspace{-2mm} 1\end{pmatrix}$}
        \end{subfigure}
        ~~
        \begin{subfigure}[b]{0.22\textwidth}
            \centering
            \includegraphics[width=\textwidth, height=0.85\textwidth]{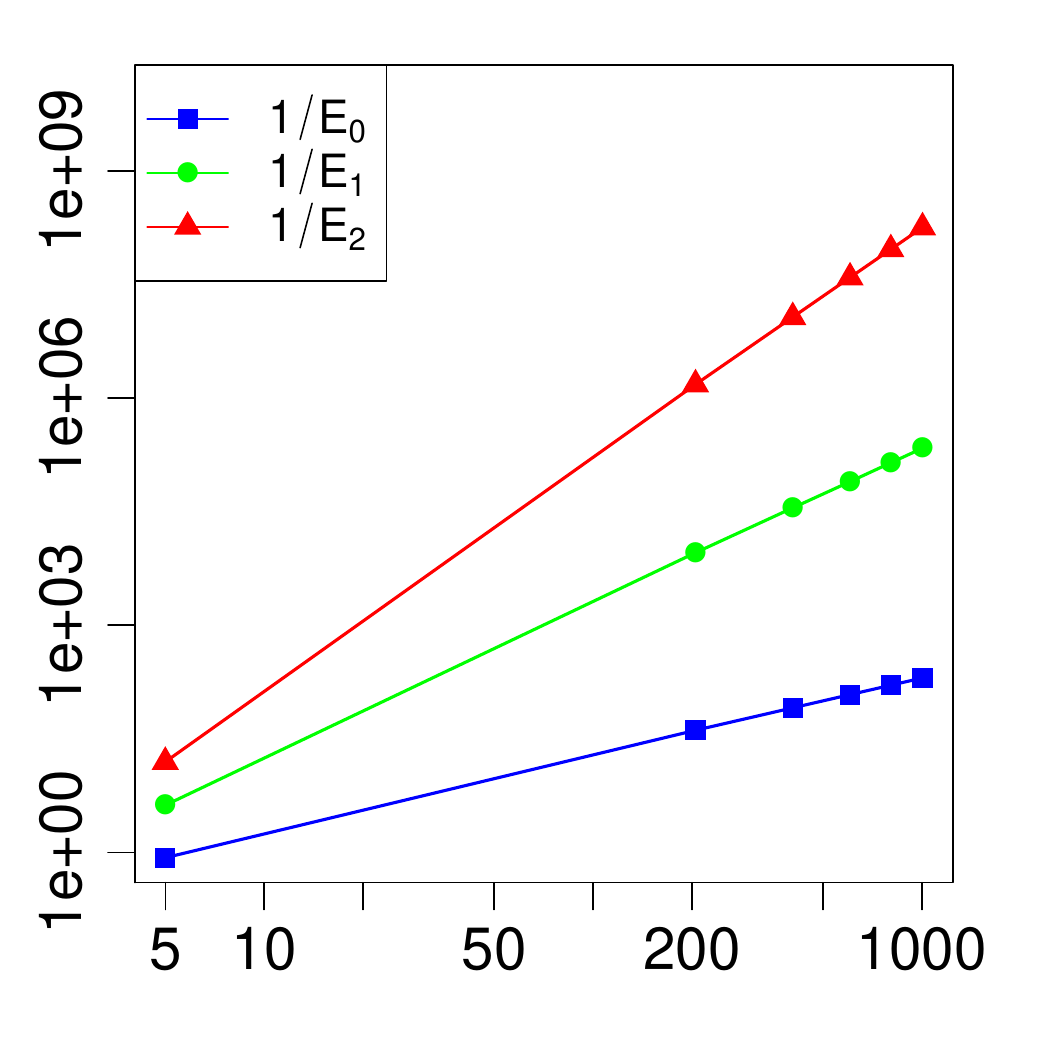}
            \vspace{-0.8cm}
            \caption{\scriptsize $\Sigma = \begin{pmatrix} 2 &\hspace{-2mm} 1 \\ 1 &\hspace{-2mm} 5\end{pmatrix}$, $\Omega = \begin{pmatrix} 1 &\hspace{-2mm} 0 \\ 0 &\hspace{-2mm} 1\end{pmatrix}$}
        \end{subfigure}
        \begin{subfigure}[b]{0.22\textwidth}
            \centering
            \includegraphics[width=\textwidth, height=0.85\textwidth]{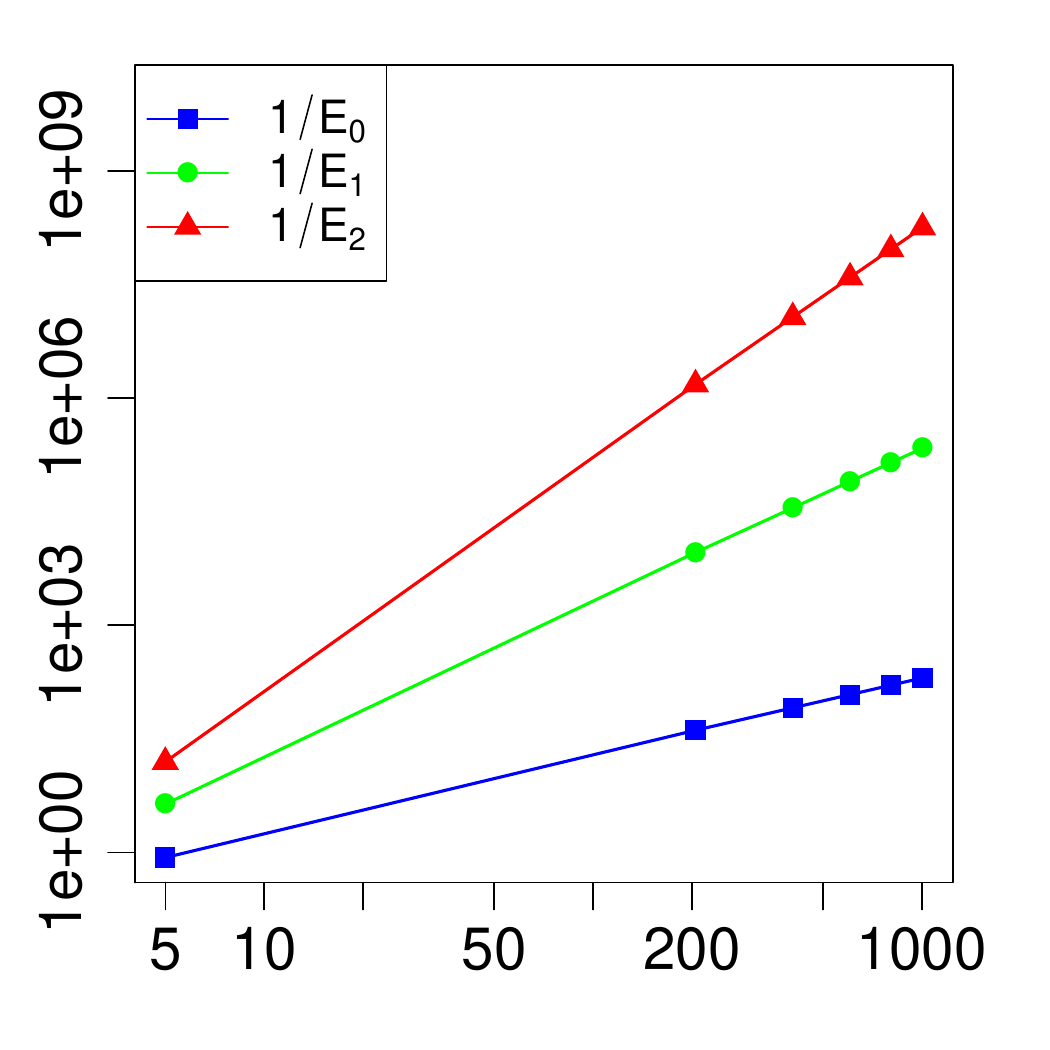}
            \vspace{-0.8cm}
            \caption{\scriptsize $\Sigma = \begin{pmatrix} 3 &\hspace{-2mm} 1 \\ 1 &\hspace{-2mm} 2\end{pmatrix}$, $\Omega = \begin{pmatrix} 1 &\hspace{-2mm} 0 \\ 0 &\hspace{-2mm} 1\end{pmatrix}$}
        \end{subfigure}
        ~~
        \begin{subfigure}[b]{0.22\textwidth}
            \centering
            \includegraphics[width=\textwidth, height=0.85\textwidth]{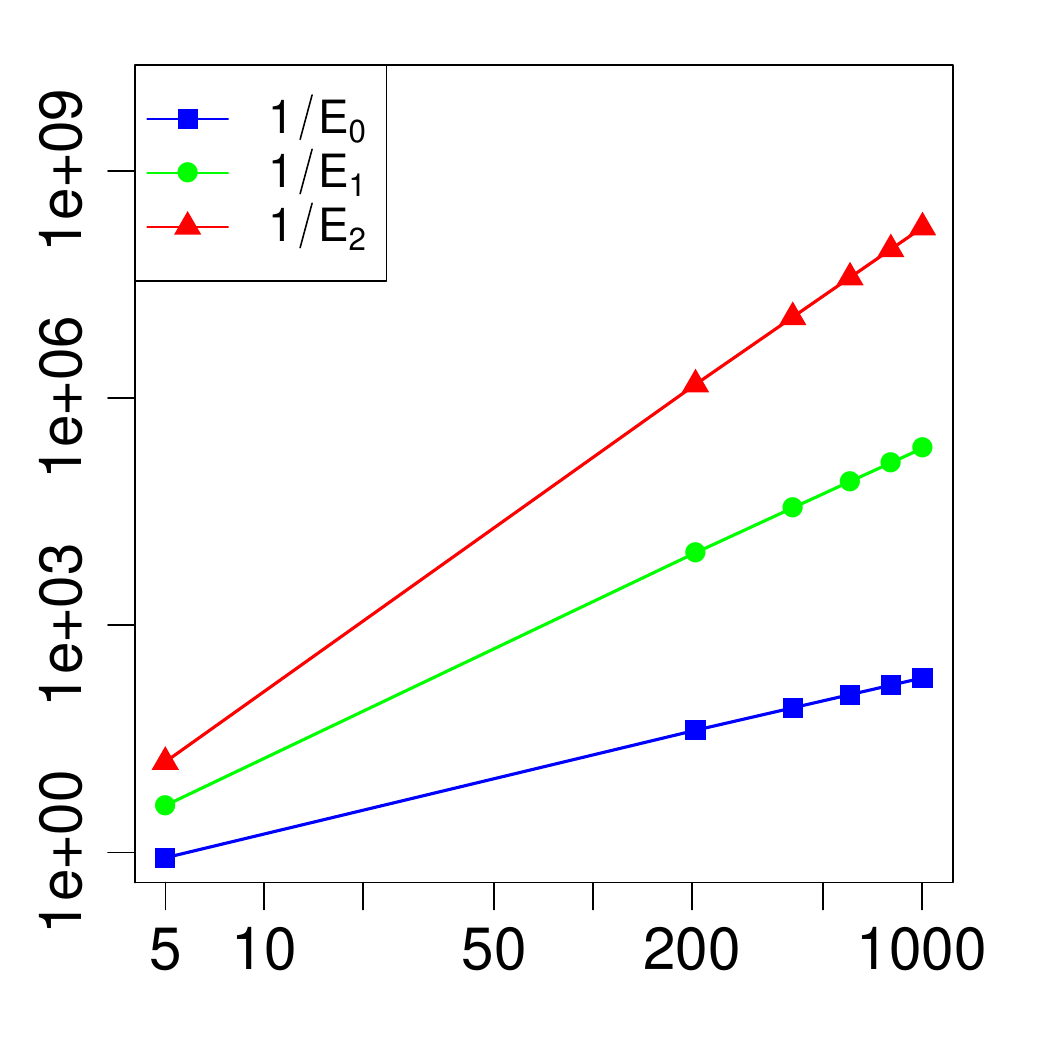}
            \vspace{-0.8cm}
            \caption{\scriptsize $\Sigma = \begin{pmatrix} 3 &\hspace{-2mm} 1 \\ 1 &\hspace{-2mm} 3\end{pmatrix}$, $\Omega = \begin{pmatrix} 1 &\hspace{-2mm} 0 \\ 0 &\hspace{-2mm} 1\end{pmatrix}$}
        \end{subfigure}
        ~~
        \begin{subfigure}[b]{0.22\textwidth}
            \centering
            \includegraphics[width=\textwidth, height=0.85\textwidth]{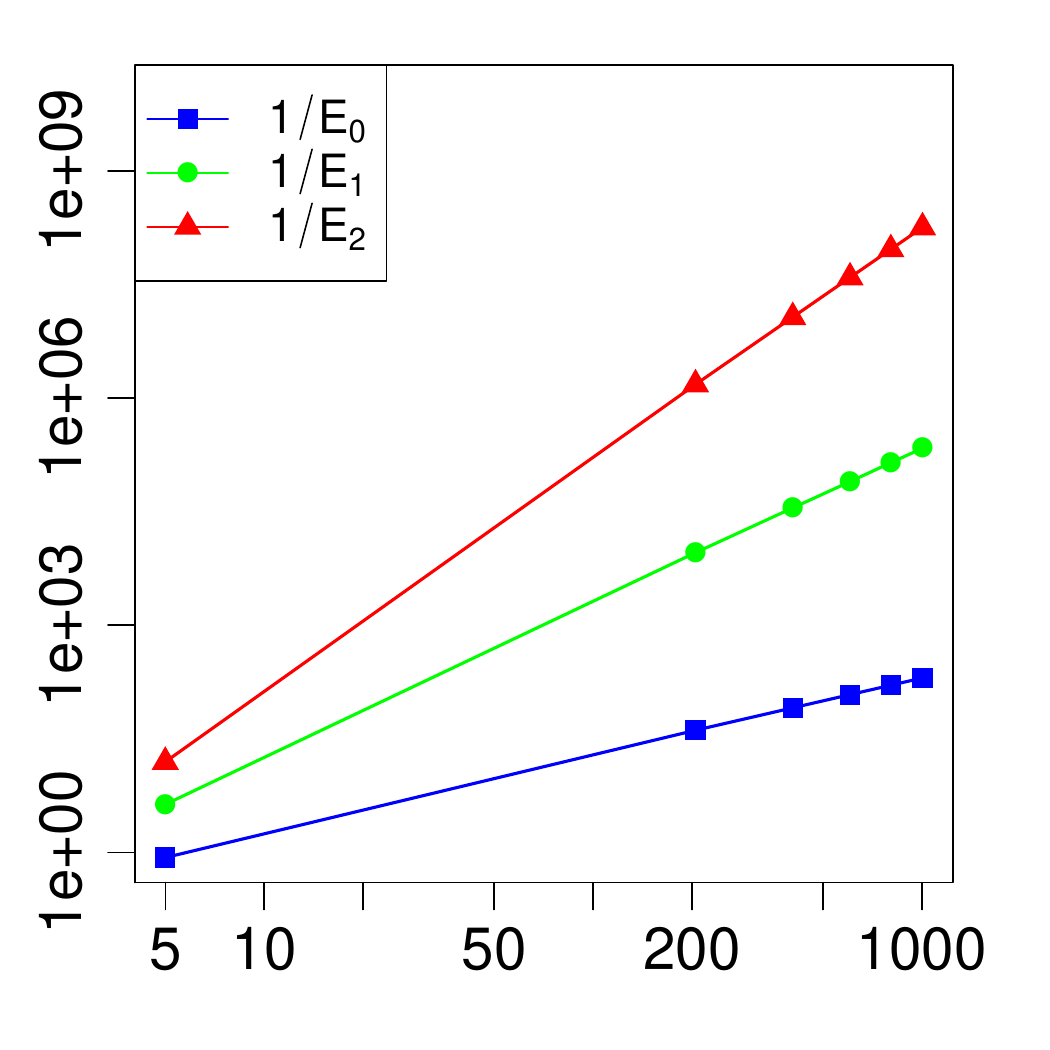}
            \vspace{-0.8cm}
            \caption{\scriptsize $\Sigma = \begin{pmatrix} 3 &\hspace{-2mm} 1 \\ 1 &\hspace{-2mm} 4\end{pmatrix}$, $\Omega = \begin{pmatrix} 1 &\hspace{-2mm} 0 \\ 0 &\hspace{-2mm} 1\end{pmatrix}$}
        \end{subfigure}
        ~~
        \begin{subfigure}[b]{0.22\textwidth}
            \centering
            \includegraphics[width=\textwidth, height=0.85\textwidth]{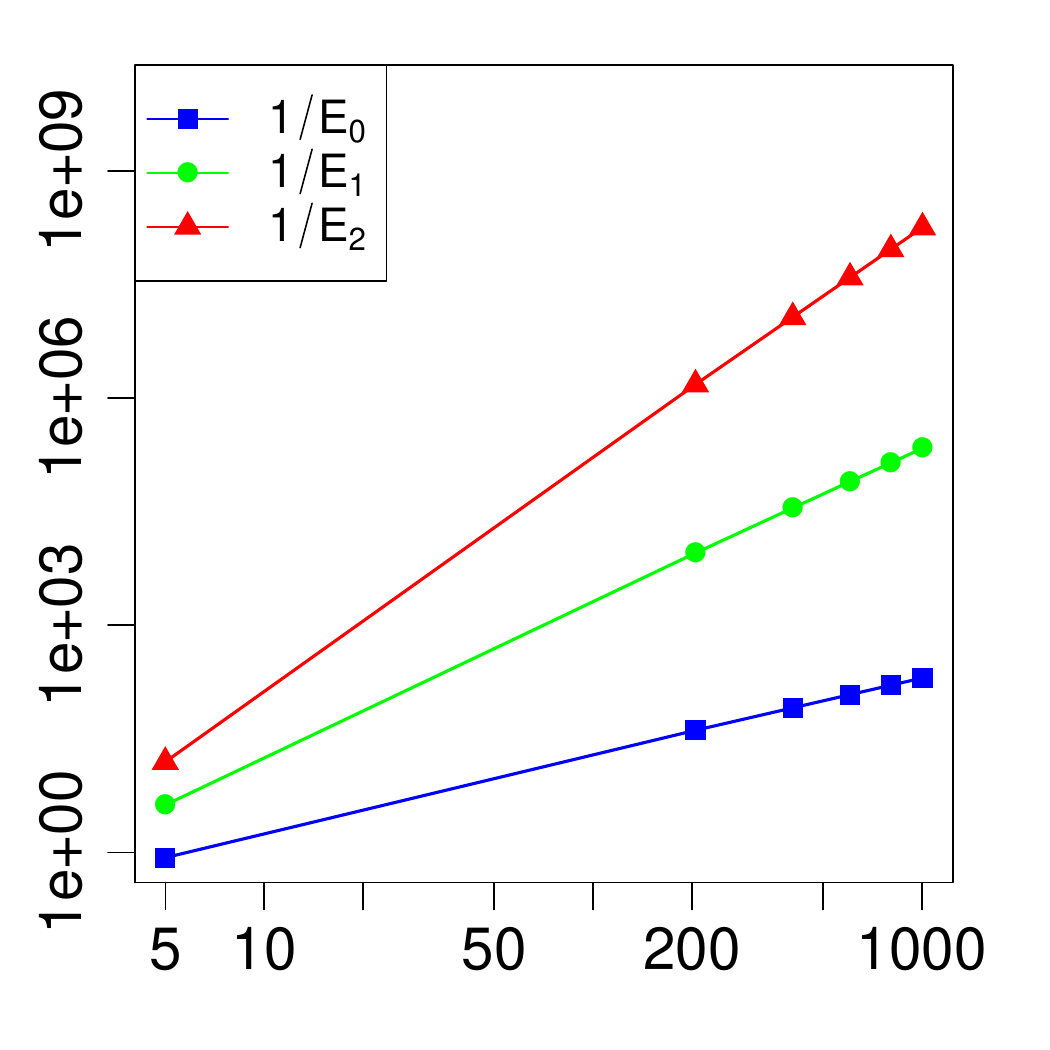}
            \vspace{-0.8cm}
            \caption{\scriptsize $\Sigma = \begin{pmatrix} 3 &\hspace{-2mm} 1 \\ 1 &\hspace{-2mm} 5\end{pmatrix}$, $\Omega = \begin{pmatrix} 1 &\hspace{-2mm} 0 \\ 0 &\hspace{-2mm} 1\end{pmatrix}$}
        \end{subfigure}
        \begin{subfigure}[b]{0.22\textwidth}
            \centering
            \includegraphics[width=\textwidth, height=0.85\textwidth]{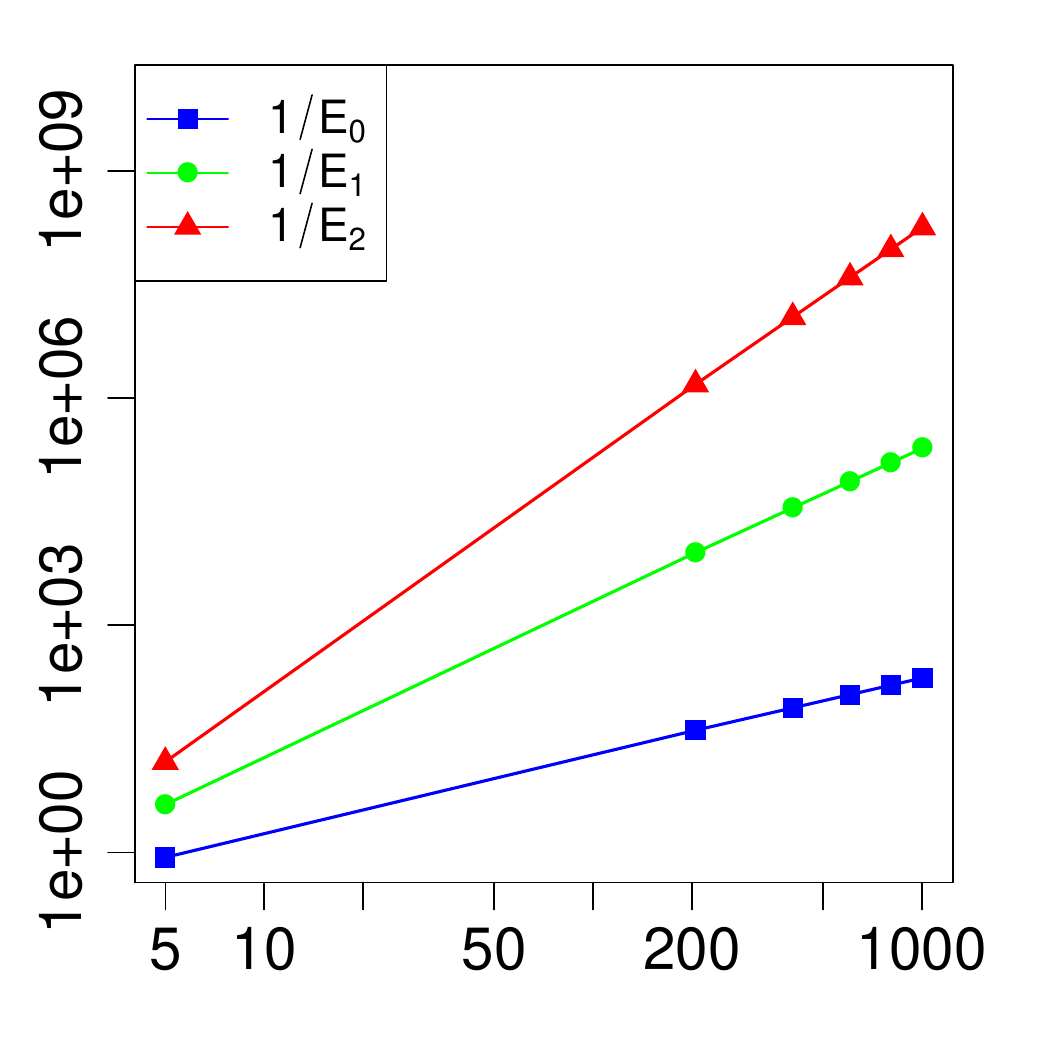}
            \vspace{-0.8cm}
            \caption{\scriptsize $\Sigma = \begin{pmatrix} 4 &\hspace{-2mm} 1 \\ 1 &\hspace{-2mm} 2\end{pmatrix}$, $\Omega = \begin{pmatrix} 1 &\hspace{-2mm} 0 \\ 0 &\hspace{-2mm} 1\end{pmatrix}$}
        \end{subfigure}
        ~~
        \begin{subfigure}[b]{0.22\textwidth}
            \centering
            \includegraphics[width=\textwidth, height=0.85\textwidth]{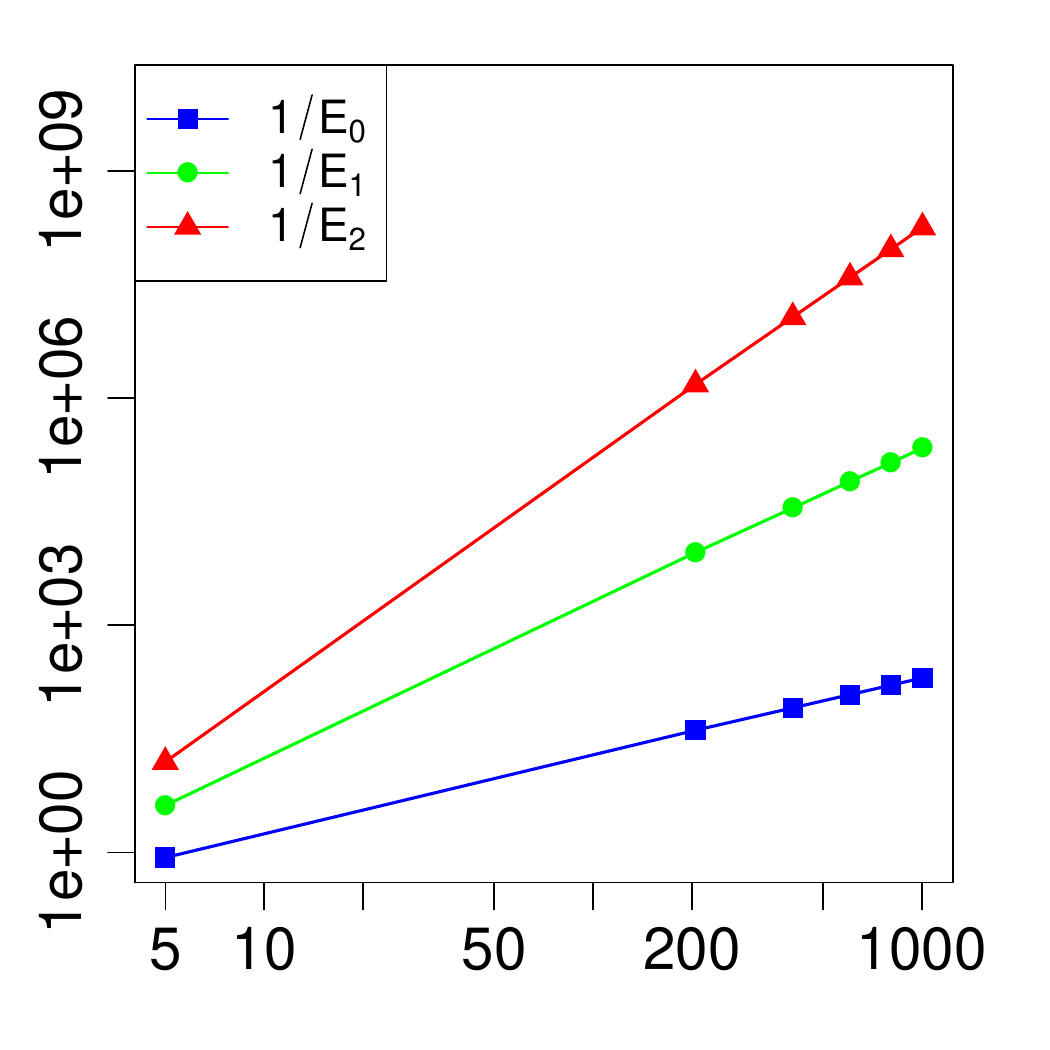}
            \vspace{-0.8cm}
            \caption{\scriptsize $\Sigma = \begin{pmatrix} 4 &\hspace{-2mm} 1 \\ 1 &\hspace{-2mm} 3\end{pmatrix}$, $\Omega = \begin{pmatrix} 1 &\hspace{-2mm} 0 \\ 0 &\hspace{-2mm} 1\end{pmatrix}$}
        \end{subfigure}
        ~~
        \begin{subfigure}[b]{0.22\textwidth}
            \centering
            \includegraphics[width=\textwidth, height=0.85\textwidth]{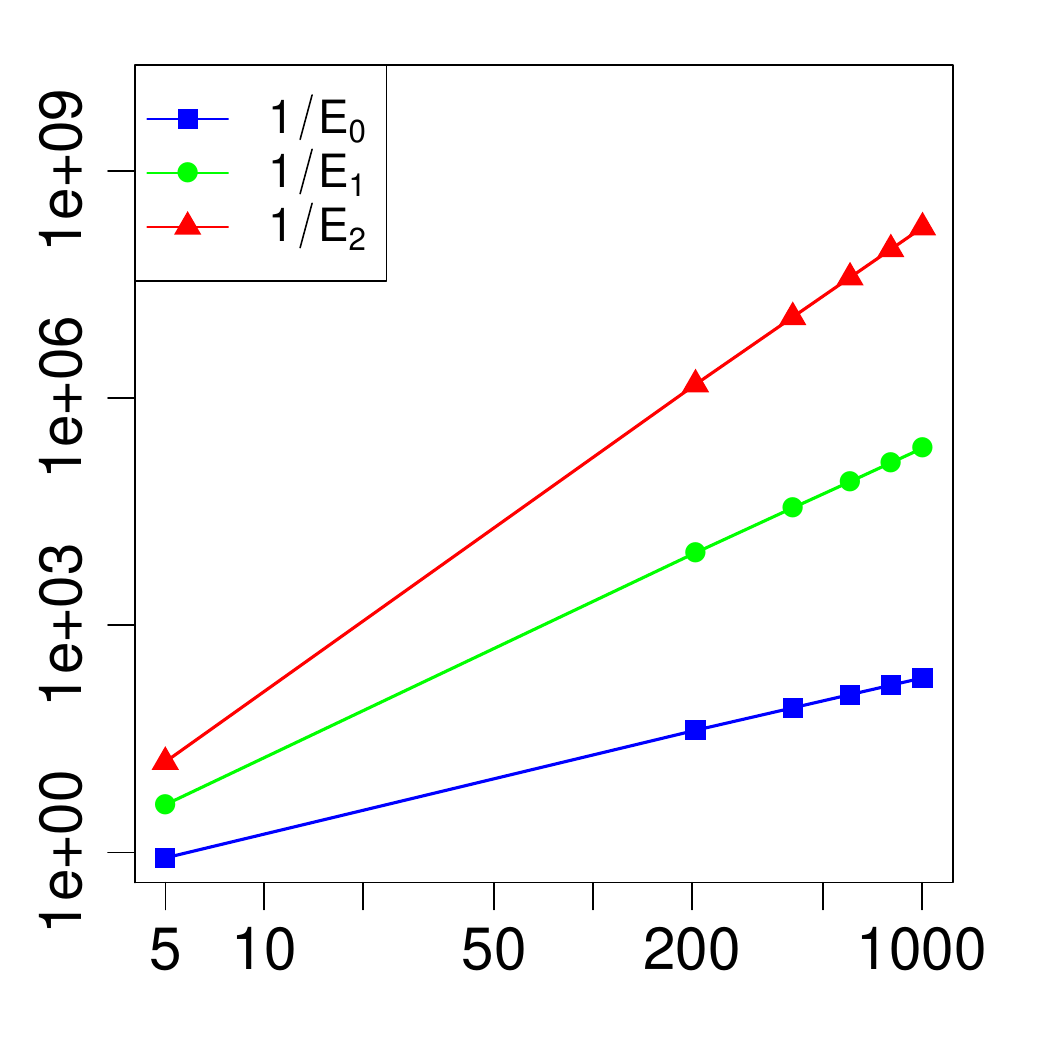}
            \vspace{-0.8cm}
            \caption{\scriptsize $\Sigma = \begin{pmatrix} 4 &\hspace{-2mm} 1 \\ 1 &\hspace{-2mm} 4\end{pmatrix}$, $\Omega = \begin{pmatrix} 1 &\hspace{-2mm} 0 \\ 0 &\hspace{-2mm} 1\end{pmatrix}$}
        \end{subfigure}
        ~~
        \begin{subfigure}[b]{0.22\textwidth}
            \centering
            \includegraphics[width=\textwidth, height=0.85\textwidth]{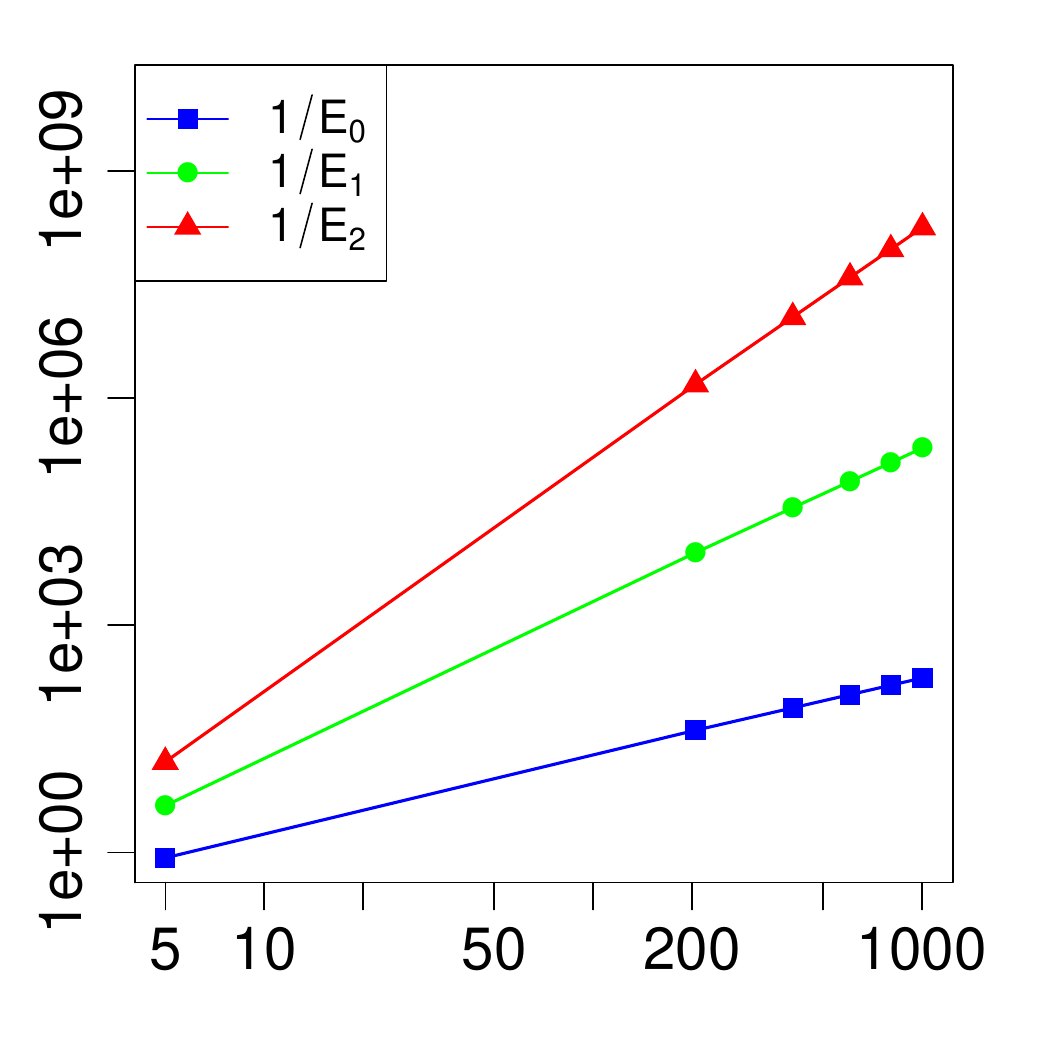}
            \vspace{-0.8cm}
            \caption{\scriptsize $\Sigma = \begin{pmatrix} 4 &\hspace{-2mm} 1 \\ 1 &\hspace{-2mm} 5\end{pmatrix}$, $\Omega = \begin{pmatrix} 1 &\hspace{-2mm} 0 \\ 0 &\hspace{-2mm} 1\end{pmatrix}$}
        \end{subfigure}
        \caption{\scriptsize Plots of $1 / E_i$ as a function of $\nu$, for various choices of $\Sigma$. Both the horizontal and vertical axes are on a logarithmic scale. The plots clearly illustrate how the addition of correction terms from Theorem~\ref{thm:LLT.matrix.T} to the base approximation \eqref{eq:E.0} improves it.}
        \label{fig:loglog.errors.plots}
    \end{figure}

\section{Proofs}\label{sec:proofs}

    \begin{proof}[\bf Proof of Theorem~\ref{thm:LLT.matrix.T}]
        First, we take the expression in \eqref{eq:matrix.T.density} over the one in \eqref{eq:sym.matrix.normal.density}:
        \begin{equation}\label{eq:proof.LLT.beginning}
            \begin{aligned}
                &\frac{[\nu / (\nu - 2)]^{md/2} \, K_{\nu,\Sigma,\Omega}(X)}{g_{\hspace{0.3mm}\Sigma,\Omega}(X / \sqrt{\nu / (\nu - 2)})} \\
                &\quad= \left[\frac{2}{\nu - 2}\right]^{md/2} \, \prod_{j=1}^d \frac{\Gamma(\frac{1}{2} (\nu + m + d - j))}{\Gamma(\frac{1}{2} (\nu + d - j))} \\
                &\quad\quad\cdot \exp\left(\frac{(\nu - 2)}{2 \nu} \mathrm{tr}\left(\Delta_X\Delta_X^{\top}\right)\right) \left|\mathrm{I}_d + \nu^{-1} \Delta_X\Delta_X^{\top}\right|^{-(\nu + m + d - 1)/2}.
            \end{aligned}
        \end{equation}
        The last determinant was obtained using the fact that the eigenvalues of a product of rectangular matrices are invariant under cyclic permutations (as long as the products remain well defined).
        Indeed, for all $j\in \{1,2,\dots,d\}$, we have
        \begin{equation}
            \begin{aligned}
                \lambda_j(\mathrm{I}_d + \nu^{-1} \Sigma^{-1} X \Omega^{-1} X^{\top})
                &= 1 + \nu^{-1} \lambda_j(\Sigma^{-1} X \Omega^{-1} X^{\top}) \\[1mm]
                &= 1 + \nu^{-1} \lambda_j(\Delta_X\Delta_X^{\top}) = \lambda_j(\mathrm{I}_d + \nu^{-1} \Delta_X\Delta_X^{\top}).
            \end{aligned}
        \end{equation}
        By taking the logarithm on both sides of \eqref{eq:proof.LLT.beginning}, we get
        \begin{equation}\label{eq:LLT.beginning.next.1}
            \begin{aligned}
                &\log \left(\frac{[\nu / (\nu - 2)]^{md/2} \, K_{\nu,\Sigma,\Omega}(X)}{g_{\hspace{0.3mm}\Sigma,\Omega}(X / \sqrt{\nu / (\nu - 2)})}\right) \\
                &\quad= -\frac{md}{2} \log\left(\frac{\nu - 2}{2}\right) + \sum_{j=1}^d \left[\log \Gamma(\tfrac{1}{2}(\nu + m + d - j)) - \log \Gamma(\tfrac{1}{2}(\nu + d - j))\right] \\[-1mm]
                &\quad\quad+ \frac{1}{2} \sum_{j=1}^d \delta_{\lambda_j}^2 - \frac{(\nu + m + d - 1)}{2} \sum_{j=1}^d \log \left(1 + \left(\frac{\delta_{\lambda_j}}{\sqrt{\nu - 2}}\right)^2\right).
            \end{aligned}
        \end{equation}
        By applying the Taylor expansions,
        \begin{equation}
            \begin{aligned}
                &\log \Gamma(\tfrac{1}{2}(\nu + m + d - j)) - \log \Gamma(\tfrac{1}{2}(\nu + d - j)) \\[1mm]
                &= \frac{1}{2} (\nu + m + d - j - 1) \log \left(\frac{1}{2} (\nu + m + d - j)\right) - \frac{1}{2} (\nu + d - j - 1) \log \left(\frac{1}{2} (\nu + d - j)\right) \\
                &\quad- \frac{m}{2} + \frac{2}{12 (\nu + m + d - j)} - \frac{2}{12 (\nu + d - j)} \\
                &\quad- \frac{2^3}{360 (\nu + m + d - j)^3} + \frac{2^3}{360 (\nu + d - j)^3} + \OO_{m,d}(\nu^{-4}) \\[2mm]
                &= \frac{m}{2} \log \left(\frac{\nu}{2}\right) + \frac{m (-2 + 2 d - 2 j + m)}{4 \nu} - \frac{m}{12 \nu^2} \left\{\hspace{-1mm}
                \begin{array}{l}
                    2 + 3 d^{\hspace{0.2mm}2} + 3 j^2 - 3 j (-2 + m) \\
                    - 3 m + m^2 + d (-6 - 6 j + 3 m)
                \end{array}
                \hspace{-1mm}\right\} \\[0.5mm]
                &\quad+ \frac{m}{24 \nu^3} \left\{\hspace{-1mm}
                \begin{array}{l}
                    4 d^{\hspace{0.2mm}3} - 4 j^3 - 6 d^{\hspace{0.2mm}2} (2 + 2 j - m) + 6 j^2 (-2 + m) + (-2 + m)^2 m \\
                    - 4 j (2 - 3 m + m^2) + 4 d (2 + 3 j^2 - 3 j (-2 + m) - 3 m + m^2)
                \end{array}
                \hspace{-1mm}\right\} + \OO_{m,d}(\nu^{-4}).
            \end{aligned}
        \end{equation}
        (see, e.g., (Ref. \cite{MR0167642}, p.~257))
        and
        \begin{equation}
            -\frac{md}{2} \log \left(\frac{\nu - 2}{2}\right) + \frac{md}{2} \log \left(\frac{\nu}{2}\right) = \frac{4md}{4 \nu} + \frac{12 md}{12 \nu^2} + \frac{32 md}{24 \nu^3} + \OO_{m,d}(\nu^{-4}),
        \end{equation}
        and
        \begin{equation}
            \log (1 + y) = y - \frac{1}{2} y^2 + \frac{1}{3} y^3 - \frac{1}{4} y^4 + \OO_{\eta}(y^5), \quad |y| < \eta < 1,
        \end{equation}
        in the above equation, we obtain
        \begin{align}\label{eq:LLT.beginning.next.2}
            &\log \left(\frac{[\nu / (\nu - 2)]^{md/2} \, K_{\nu,\Sigma,\Omega}(X)}{g_{\hspace{0.3mm}\Sigma,\Omega}(X / \sqrt{\nu / (\nu - 2)})}\right) \notag \\
            &= \sum_{j=1}^d \frac{m (2 + 2 d - 2 j + m)}{4 \nu} - \sum_{j=1}^d \frac{m}{12 \nu^2}
                \left\{\hspace{-1mm}
                \begin{array}{l}
                    -10 + 3 d^{\hspace{0.2mm}2} + 3 j^2 - 3 j (-2 + m) \\
                    - 3 m + m^2 + d (-6 - 6 j + 3 m)
                \end{array}
                \hspace{-1mm}\right\} \notag \\[0.5mm]
            &\quad+ \sum_{j=1}^d \frac{m}{24 \nu^3} \left\{\hspace{-1mm}
                \begin{array}{l}
                    32 + 4 d^{\hspace{0.2mm}3} - 4 j^3 - 6 d^{\hspace{0.2mm}2} (2 + 2 j - m) + 6 j^2 (-2 + m) + (-2 + m)^2 m \\
                    - 4 j (2 - 3 m + m^2) + 4 d (2 + 3 j^2 - 3 j (-2 + m) - 3 m + m^2)
                \end{array}
                \hspace{-1mm}\right\} \notag \\
            &\quad+ \frac{1}{2} \sum_{j=1}^d \delta_{\lambda_j}^2 - \frac{(\nu + m + d - 1)}{2} \sum_{j=1}^d \left(\frac{\delta_{\lambda_j}}{\sqrt{\nu - 2}}\right)^2 \notag \\
            &\quad\quad+ \frac{(\nu + m + d - 1)}{4} \sum_{j=1}^d \left(\frac{\delta_{\lambda_j}}{\sqrt{\nu - 2}}\right)^4 - \frac{(\nu + m + d - 1)}{6} \sum_{j=1}^d \left(\frac{\delta_{\lambda_j}}{\sqrt{\nu - 2}}\right)^6 \notag \\
            &\quad\quad+ \frac{(\nu + m + d - 1)}{8} \sum_{j=1}^d \left(\frac{\delta_{\lambda_j}}{\sqrt{\nu - 2}}\right)^8 + \OO_{d,m,\eta}\left(\frac{1 + \max_{1 \leq j \leq d} |\delta_{\lambda_j}|^{10}}{\nu^4}\right).
        \end{align}
        Now,
        \begin{equation}
            \begin{aligned}
                \frac{1}{2} - \frac{\nu + m + d - 1}{2 (\nu - 2)} &= -\frac{(m + d + 1)}{2\nu} - \frac{(m + d + 1)}{\nu^2} - \frac{2 (m + d + 1)}{\nu^3} + \OO_{m,d}(\nu^{-4}), \\
                \frac{\nu + m + d - 1}{4 (\nu - 2)^2} &= \frac{1}{4 \nu} + \frac{(m + d + 3)}{4 \nu^2} + \frac{(m + d + 2)}{\nu^3} + \OO_{m,d}(\nu^{-4}), \\
                -\frac{\nu + m + d - 1}{6 (\nu - 2)^3} &= -\frac{1}{6 \nu^2} - \frac{(m + d + 5)}{6 \nu^3} + \OO_{m,d}(\nu^{-4}), \\
                \frac{\nu + m + d - 1}{8 (\nu - 2)^4} &= \frac{1}{8 \nu^3} + \OO_{m,d}(\nu^{-4}),
            \end{aligned}
        \end{equation}
        so we can rewrite \eqref{eq:LLT.beginning.next.2} as
        \begin{equation}\label{eq:LLT.beginning.next.3}
            \begin{aligned}
                &\log \left(\frac{[\nu / (\nu - 2)]^{md/2} \, K_{\nu,\Sigma,\Omega}(X)}{g_{\hspace{0.3mm}\Sigma,\Omega}(X / \sqrt{\nu / (\nu - 2)})}\right) \\
                &= \nu^{-1} \sum_{j=1}^d \left\{\frac{1}{4} \delta_{\lambda_j}^4 - \frac{(m + d + 1)}{2} \delta_{\lambda_j}^2 + \frac{m (2 + 2 d - 2 j + m)}{4}\right\} \\
                &\quad+ \nu^{-2} \sum_{j=1}^d \left\{\hspace{-1mm}
                    \begin{array}{l}
                        -\frac{1}{6} \delta_{\lambda_j}^6 + \frac{(m + d + 3)}{4} \delta_{\lambda_j}^4 - (m + d + 1) \delta_{\lambda_j}^2 \\[1mm]
                        - \frac{m}{12}
                        \left\{\hspace{-1mm}
                        \begin{array}{l}
                            -10 + 3 d^{\hspace{0.2mm}2} + 3 j^2 - 3 j (-2 + m) \\
                            - 3 m + m^2 + d (-6 - 6 j + 3 m)
                        \end{array}
                        \hspace{-1mm}\right\}
                    \end{array}
                    \hspace{-1mm}\right\} \\
                &\quad+ \nu^{-3} \sum_{j=1}^d \left\{\hspace{-1mm}
                    \begin{array}{l}
                        \frac{1}{8} \delta_{\lambda_j}^8 - \frac{(m + d + 5)}{6} \delta_{\lambda_j}^6 + (m + d + 2) \delta_{\lambda_j}^4 - 2 (m + d + 1) \delta_{\lambda_j}^2 \\[1mm]
                        + \frac{m}{24} \left\{\hspace{-1mm}
                        \begin{array}{l}
                            32 + 4 d^{\hspace{0.2mm}3} - 4 j^3 - 6 d^{\hspace{0.2mm}2} (2 + 2 j - m) + 6 j^2 (-2 + m) + (-2 + m)^2 m \\
                            - 4 j (2 - 3 m + m^2) + 4 d (2 + 3 j^2 - 3 j (-2 + m) - 3 m + m^2)
                        \end{array}
                        \hspace{-1mm}\right\}
                    \end{array}
                    \hspace{-1mm}\right\} \\
                &\quad+ \OO_{d,m,\eta}\left(\frac{1 + \max_{1 \leq j \leq d} |\delta_{\lambda_j}|^{10}}{\nu^4}\right),
            \end{aligned}
        \end{equation}
        which proves \eqref{eq:LLT.order.2.log} after some simplifications with \texttt{Mathematica}.
    \end{proof}

    \begin{proof}[\bf Proof of Theorem~\ref{thm:probability.metric.bounds}]
        By the comparison of the total variation norm $\|\cdot\|$ with the Hellinger distance on page 726 of \citet{MR1922539}, we already know that
        \begin{equation}\label{eq:first.bound.probability.metric.bounds}
            \begin{aligned}
                &\|\PP_{\nu,\Sigma,\Omega} - \QQ_{\Sigma,\Omega}\| \\
                &\quad\leq \sqrt{2 \, \PP\left(X\in B_{\nu,\Sigma,\Omega}^{\hspace{0.3mm}c}(1/2)\right) + \EE\left[\log\Bigg(\frac{\rd \PP_{\nu,\Sigma,\Omega}}{\rd \QQ_{\Sigma,\Omega}}(X)\Bigg) \, \ind_{\{X\in B_{\nu,\Sigma,\Omega}(1/2)\}}\right]}.
            \end{aligned}
        \end{equation}
        Given that $\Delta_X = \Sigma^{-1/2} X \Omega^{-1/2}\sim T_{d,m}(\nu,\mathrm{I}_d,\mathrm{I}_m)$ by Theorem~4.3.5 in \cite{Gupta_Nagar_1999}, we know, by Theorem~4.2.1 in \cite{Gupta_Nagar_1999}, that
        \begin{equation}
            \Delta_X \stackrel{\mathrm{law}}{=} (\nu^{-1} \mat{S})^{-1/2} Z,
        \end{equation}
        for $\mat{S}\sim \mathrm{Wishart}_{d\times d}(\nu + d - 1, \mathrm{I}_d)$ and $Z\sim \mathrm{MN}_{d\times m}(0_{d\times m}, \mathrm{I}_d \otimes \mathrm{I}_m)$ that are independent,
        so that, by Theorems~3.3.1 and 3.3.3~in \cite{Gupta_Nagar_1999}, we have
        \begin{equation}\label{eq:conditional}
            \Delta_X \Delta_X^{\top} \nvert \mat{S}\sim \mathrm{Wishart}_{d\times d} (m, \nu \, \mat{S}^{-1}).
        \end{equation}
        Therefore, by conditioning on $\mat{S}$, and then by applying the sub-multiplicativity of the largest eigenvalue for nonnegative definite matrices, and a large deviation bound on the maximum eigenvalue of a Wishart matrix (which is sub-exponential), we get, for $\nu$ large enough,
        \begin{align}\label{eq:concentration.bound}
            \PP\left(X\in B_{\nu,\Sigma,\Omega}^{\hspace{0.3mm}c}(1/2)\right)
            &\leq \EE\left[\left.\PP\left(\lambda_1(\Delta_X \Delta_X^{\top}) > \frac{\nu^{1/2}}{4}~\right| \, \mat{S}\right)\right] \notag \\
            &\leq \EE\left[\left.\PP\left(\lambda_1((\nu^{-1} \mat{S})^{-1/2}) \lambda_1(Z Z^{\top}) \lambda_1((\nu^{-1} \mat{S})^{-1/2}) > \frac{\nu^{1/2}}{4}~\right| \, \mat{S}\right)\right] \notag \\
            &= \EE\left[\left.\PP\left(\lambda_1(Z Z^{\top}) > \frac{\lambda_d(\mat{S})}{4 \, \nu^{1/2}}~\right| \, \mat{S}\right)\right] \notag \\[-0.5mm]
            &\leq C_{m,d} \, \exp\left(- \frac{\nu^{1/2}}{10^4 m d}\right),
        \end{align}
        for some positive constant $C_{m,d}$ that depends only on $m$ and $d$.
        By Theorem~\ref{thm:LLT.matrix.T}, we also have
        \begin{equation}\label{eq:estimate.I.begin}
            \begin{aligned}
                &\EE\left[\log\Bigg(\frac{\rd \PP_{\nu,\Sigma,\Omega}}{\rd \QQ_{\Sigma,\Omega}}(X)\Bigg) \, \ind_{\{X\in B_{\nu,\Sigma,\Omega}(1/2)\}}\right] \\
                &\quad= \nu^{-1} \left\{\hspace{-1mm}
                    \begin{array}{l}
                        \frac{1}{4} \cdot \EE\left[\mathrm{tr}\left((\Delta_X \Delta_X^{\top})^2\right)\right] \\
                        - \frac{(m + d + 1)}{2} \cdot \EE\left[\mathrm{tr}(\Delta_X \Delta_X^{\top})\right] + \frac{m d (m + d + 1)}{4}
                    \end{array}
                    \hspace{-1mm}\right\} \\[0.5mm]
                &\qquad+ \nu^{-1} \left\{\hspace{-1mm}
                    \begin{array}{l}
                        \OO\left(\EE\left[\mathrm{tr}\left((\Delta_X \Delta_X^{\top})^2\right) \ind_{\{X\in B_{\nu,\Sigma,\Omega}(1/2)\}}\right]\right) \\
                        + (m + d) \, \OO\left(\EE\left[\mathrm{tr}(\Delta_X \Delta_X^{\top}) \ind_{\{X\in B_{\nu,\Sigma,\Omega}(1/2)\}}\right]\right) + \OO(m (m + d))
                    \end{array}
                    \hspace{-1mm}\right\} \\[0.5mm]
                &\qquad+ \nu^{-2} \left\{\hspace{-1mm}
                    \begin{array}{l}
                        \OO\left(\EE\left[\mathrm{tr}\left((\Delta_X \Delta_X^{\top})^3\right)\right]\right) + (m + d) \, \OO\left(\EE\left[\mathrm{tr}\left((\Delta_X \Delta_X^{\top})^2\right)\right]\right) \\[1.5mm]
                        + (m + d) \, \OO\left(\EE\left[\mathrm{tr}(\Delta_X \Delta_X^{\top})\right]\right) + \OO(m d (m + d)^2)
                    \end{array}
                    \hspace{-1mm}\right\}.
            \end{aligned}
        \end{equation}
        On the right-hand side, the first line is estimated using Lemma~\ref{lem:Leblanc.2012.boundary.Lemma.1}, and the second line is bounded using Lemma~\ref{lem:Leblanc.2012.boundary.Lemma.1.with.set.A}.
        We find
        \begin{equation}\label{eq:estimate.I.begin.next}
            \EE\left[\log\Bigg(\frac{\rd \PP_{\nu,\Sigma,\Omega}}{\rd \QQ_{\Sigma,\Omega}}(X)\Bigg) \, \ind_{\{X\in B_{\nu,\Sigma,\Omega}(1/2)\}}\right] = \OO(m^3 d^{\hspace{0.2mm}3} \nu^{-2}).
        \end{equation}
        Putting \eqref{eq:concentration.bound} and \eqref{eq:estimate.I.begin} together in \eqref{eq:first.bound.probability.metric.bounds} gives the conclusion.
    \end{proof}

\vspace{5mm}
\noindent
{\bf Supplementary Materials:} The \texttt{R} code for the simulations in Section~\ref{sec:main.results} can be downloaded at: \href{https://www.mdpi.com/article/10.3390/appliedmath0000000/s1}{https://www.mdpi.com/article/10.3390/appliedmath0000000/s1}.

\vspace{2mm}
\noindent
{\bf Funding:} F.O. is supported by postdoctoral fellowships from the NSERC (PDF) and the FRQNT (B3X supplement and B3XR).

\vspace{2mm}
\noindent
{\bf Acknowledgments:} We thank the three referees for their comments.

\vspace{2mm}
\noindent
{\bf Conflicts of Interest:} The author declares no conflicts of interest.

\appendix

\section{Technical computations}\label{sec:technical.computations}

    Below, we compute the expectations for some traces of powers of the matrix-variate Student distribution.
    The lemma is used to estimate some trace moments and the $\asymp \nu^{-2}$ errors in \eqref{eq:estimate.I.begin} of the proof of Theorem~\ref{thm:probability.metric.bounds}, and also as a preliminary result for the proof of Lemma~\ref{lem:Leblanc.2012.boundary.Lemma.1.with.set.A}.

    \begin{lemma}\label{lem:Leblanc.2012.boundary.Lemma.1}
        Let $d,m\in \N$, $\Sigma\in \mathcal{S}_{++}^{\hspace{0.3mm}d}$ and $\Omega\in \mathcal{S}_{++}^{\hspace{0.3mm}m}$ be given.
        If $X\sim T_{d,m}(\nu,\Sigma,\Omega)$ according to~\eqref{eq:matrix.T.density}, then
        \vspace{-2mm}
        \begin{align}
            \EE\left[\mathrm{tr}\left(\Delta_X \Delta_X^{\top}\right)\right]
            &= \frac{m d \, \nu}{\nu - 2}, \label{eq:trace.1} \\[1mm]
            \EE\left[\mathrm{tr}\left((\Delta_X \Delta_X^{\top})^2\right)\right]
            &= \frac{m d \, \nu^2 \, \left\{(m + d) (\nu - 2) + \nu + m d\right\}}{(\nu - 1) (\nu - 2) (\nu - 4)}, \label{eq:trace.2}
        \end{align}
        where we recall $\Delta_X \leqdef \Sigma^{-1/2} X \Omega^{-1/2}$.
        In particular, as $\nu\to \infty$, we have
        \begin{align}
            \EE\left[\mathrm{tr}\left(\Delta_X \Delta_X^{\top}\right)\right] \sim m d \quad \text{and} \quad \EE\left[\mathrm{tr}\left((\Delta_X \Delta_X^{\top})^2\right)\right] \sim m d (m + d + 1).
        \end{align}
    \end{lemma}

    \begin{proof}[{\bf Proof of Lemma~\ref{lem:Leblanc.2012.boundary.Lemma.1}}]
        For $\mat{W}\sim \mathrm{Wishart}_{d\times d}(n, \mat{V})$ with $n > 0$ and $\mat{V}\in \mathcal{S}_{++}^d$, we know from (Ref. \cite{Gupta_Nagar_1999}, p.~99) (alternatively, see (Ref. \cite{MR347003}, p.~66) or (Ref. \cite{MR2066255}, p.~308)) that
        \begin{equation}
            \EE[\mat{W}] = n \, \mat{V} \quad \text{and} \quad \EE[\mat{W}^2] = n \, \left\{(n + 1) \, \mat{V} + \mathrm{tr}(\mat{V}) \, \mathrm{I}_d\right\} \, \mat{V},
        \end{equation}
        and from (Ref. \cite{Gupta_Nagar_1999}, pp.~99--100) (alternatively, see \cite{MR556910} and (\cite{MR2066255}, p.~308), or (\cite{MR968156}, pp.~101--103)) that
        \begin{align}
            &\EE[\mat{W}^{-1}] = \frac{\mat{V}}{n - d - 1}, \quad \text{for } n - d - 1 > 0, \\[1mm]
            &\EE[\mat{W}^{-2}] = \frac{\mathrm{tr}(\mat{V}^{-1}) \, \mat{V}^{-1} + (n - d - 1) \, \mat{V}^{-2}}{(n - d) (n - d - 1) (n - d - 3)}, \quad \text{for } n - d - 3 > 0,
        \end{align}
        and from (Corollary~3.1 in \cite{MR968156}) that
        \begin{equation}
            \EE[\mathrm{tr}(\mat{W}^{-1}) \, \mat{W}^{-1}] = \frac{(n - d - 2) \, \mathrm{tr}(\mat{V}^{-1}) \, \mat{V}^{-1} + 2 \, \mat{V}^{-2}}{(n - d) (n - d - 1) (n - d - 3)}, \quad \text{for } n - d - 3 > 0.
        \end{equation}
        Therefore, by combining the above moment estimates with \eqref{eq:conditional}, we have
        \begin{align}
            \EE\left[\Delta_X \Delta_X^{\top}\right]
            &= \EE\left[\EE[\Delta_X \Delta_X^{\top} \nvert \mat{S}]\right] = \EE[m \, (\nu \, \mat{S}^{-1})] = m \, \nu \, \EE[\mat{S}^{-1}] = \frac{m \, \nu}{\nu - 2} \, \mathrm{I}_d, \\[1mm]
            \EE\left[(\Delta_X \Delta_X^{\top})^2\right]
            &= \EE\left[\EE[(\Delta_X \Delta_X^{\top})^2 \nvert \mat{S}]\right] = \EE\left[m \, \left\{(m + 1) \, (\nu \, \mat{S}^{-1}) + \mathrm{tr}(\nu \, \mat{S}^{-1}) \, \mathrm{I}_d\right\} \, (\nu \, \mat{S}^{-1})\right] \notag \\
            &= m \, \nu^2 \, \left\{(m + 1) \, \EE[\mat{S}^{-2}] + \EE[\mathrm{tr}(\mat{S}^{-1}) \, \mat{S}^{-1}]\right\} \notag \\
            &= \frac{m \, \nu^2 \, \left\{(m + 1) \, (\nu + d - 2)  + (\nu - 3) \, d + 2\right\}}{(\nu - 1) (\nu - 2) (\nu - 4)} \, \mathrm{I}_d,
        \end{align}
        By linearity, the trace of an expectation is the expectation of the trace, so \eqref{eq:trace.1} and \eqref{eq:trace.2} follow from the above equations.
    \end{proof}

    We can also estimate the moments of Lemma~\ref{lem:Leblanc.2012.boundary.Lemma.1} on various events.
    The lemma below is used to estimate the $\asymp \nu^{-1}$ errors in \eqref{eq:estimate.I.begin} of the proof of Theorem~\ref{thm:probability.metric.bounds}.

    \begin{lemma}\label{lem:Leblanc.2012.boundary.Lemma.1.with.set.A}
        Let $d,m\in \N$, $\Sigma\in \mathcal{S}_{++}^{\hspace{0.3mm}d}$ and $\Omega\in \mathcal{S}_{++}^{\hspace{0.3mm}m}$ be given, and let $A\in \mathscr{B}(\R^{d\times m})$ be a Borel set.
        If $X\sim T_{d,m}(\nu,\Sigma,\Omega)$ according to \eqref{eq:matrix.T.density}, then, for $\nu$ large enough,
        \begin{align}
            &\left|\EE\left[\mathrm{tr}(\Delta_X \Delta_X^{\top}) \ind_{\{X\in A\}}\right]\right| \leq 2 \, m d^{\hspace{0.3mm}3/2} \, \left(\PP\left(X\in A^c\right)\right)^{1/2}\hspace{-0.5mm}, \label{eq:thm:central.moments.eq.1.set.A} \\[1mm]
            &\left|\EE\left[\mathrm{tr}\left((\Delta_X \Delta_X^{\top})^2\right) \ind_{\{X\in A\}}\right] - \frac{m d \, \nu^2 \left\{(m + d) (\nu - 2) + \nu + m d\right\}}{(\nu - 1) (\nu - 2) (\nu - 4)}\right| \notag \\
            &\qquad\leq 100 \, m^2 d^{\hspace{0.3mm}5/2} \, \left(\PP\left(X\in A^c\right)\right)^{1/2}\hspace{-0.5mm}\hspace{-0.5mm}, \label{eq:thm:central.moments.eq.2.set.A}
        \end{align}
        where we recall $\Delta_X \leqdef \Sigma^{-1/2} X \Omega^{-1/2}$.
    \end{lemma}

    \begin{proof}[{\bf Proof of Lemma~\ref{lem:Leblanc.2012.boundary.Lemma.1.with.set.A}}]
        By Lemma~\ref{lem:Leblanc.2012.boundary.Lemma.1}, the Cauchy--Schwarz inequality and Jensen's inequality,
        \begin{equation}
            (\mathrm{tr}(\Delta_X \Delta_X^{\top}))^2 \leq d \cdot \mathrm{tr}((\Delta_X \Delta_X^{\top})^2),
        \end{equation}
        we have
        \begin{equation}
            \begin{aligned}
                \left|\EE\left[\mathrm{tr}(\Delta_X \Delta_X^{\top}) \ind_{\{X\in A\}}\right]\right|
                &= \left|\EE\left[\mathrm{tr}(\Delta_X \Delta_X^{\top}) \ind_{\{X\in A^c\}}\right]\right| \\
                &\leq \left(\EE\left[(\mathrm{tr}(\Delta_X \Delta_X^{\top}))^2\right]\right)^{1/2} \left(\PP\left(X\in A^c\right)\right)^{1/2} \\[1mm]
                &\leq \left(d \cdot \EE\left[\mathrm{tr}((\Delta_X \Delta_X^{\top})^2)\right]\right)^{1/2} \left(\PP\left(X\in A^c\right)\right)^{1/2} \\
                &\leq 2 \, m d^{\hspace{0.3mm}3/2} \, \left(\PP\left(X\in A^c\right)\right)^{1/2}\hspace{-0.5mm},
            \end{aligned}
        \end{equation}
        which proves \eqref{eq:thm:central.moments.eq.1.set.A}.
        Similarly, by Lemma~\ref{lem:Leblanc.2012.boundary.Lemma.1}, Holder's inequality and Jensen's inequality,
        \begin{equation}
            (\mathrm{tr}((\Delta_X \Delta_X^{\top})^2))^2 \leq d \, \mathrm{tr}((\Delta_X \Delta_X^{\top})^4),
        \end{equation}
        we have, for $\nu$ large enough,
        \begin{equation}
            \begin{aligned}
                &\left|\EE\left[\mathrm{tr}((\Delta_X \Delta_X^{\top})^2) \ind_{\{X\in A\}}\right] - \frac{m d \, \nu^2 \, \left\{(m + d) (\nu - 2) + \nu + m d\right\}}{(\nu - 1) (\nu - 2) (\nu - 4)}\right| \\
                &\quad= \left|\EE\left[\mathrm{tr}((\Delta_X \Delta_X^{\top})^2) \ind_{\{X\in A^c\}}\right]\right| \leq \left(\EE\left[(\mathrm{tr}((\Delta_X \Delta_X^{\top})^2))^2\right]\right)^{1/2} \left(\PP\left(X\in A^c\right)\right)^{1/2} \\[1mm]
                &\quad\leq \left(d \, \EE\left[\mathrm{tr}((\Delta_X \Delta_X^{\top})^4)\right]\right)^{1/2} \left(\PP\left(X\in A^c\right)\right)^{1/2} \\
                &\quad\leq \left(d \, 10^4 (m d)^4\right)^{1/2} \, \left(\PP\left(X\in A^c\right)\right)^{1/2} \leq 100 \, m^2 d^{\hspace{0.3mm}5/2} \, \left(\PP\left(X\in A^c\right)\right)^{1/2}\hspace{-0.5mm},
            \end{aligned}
        \end{equation}
        which proves \eqref{eq:thm:central.moments.eq.2.set.A}.
        This ends the proof.
    \end{proof}


\bibliographystyle{authordate1}

\begin{thebibliography}{}

\end{thebibliography}


\begin{thebibliography}{999}

\bibitem[Gupta and Nagar(1999)]{Gupta_Nagar_1999}
Gupta, A.K.; Nagar, D.K.
\newblock {\em Matrix {V}ariate {D}istributions}, 1st ed.; Chapman and
  Hall/CRC: Boca Raton, FL, USA, 1999; p. 384.

\bibitem[Olver \em{et~al.}(2010)Olver, Lozier, Boisvert, and Clark]{MR2723248}
Olver, F.W.J.; Lozier, D.W.; Boisvert, R.F.; Clark, C.W. (Eds.)
\newblock {\em N{IST} {H}andbook of {M}athematical {F}unctions}; U.S.
  Department of Commerce, National Institute of Standards and Technology:
  Washington, DC, USA; Cambridge University Press: Cambridge, UK, 2010; pp. xvi+951.

\bibitem[Nagar {et~al.}(2013)Nagar, Rold\'{a}n-Correa, and Gupta]{MR3189307}
Nagar, D.K.; Rold\'{a}n-Correa, A.; Gupta, A.K.
\newblock Extended matrix variate gamma and beta functions.
\newblock {\em J. Multivar. Anal.} {\bf 2013}, {\em 122},~53--69.

\bibitem[Pajevic and Basser(1999)]{Pajevic_Basser_1999}
Pajevic, S.; Basser, P.J.
\newblock Parametric description of noise in diffusion tensor {MRI}.
In Proceedings of the 7th Annual Meeting of the ISMRM, Philadelphia, PA, USA, 22--28 May 1999; p. 1787.

\bibitem[Basser and Jones(2002)]{doi:10.1002nbm.783}
Basser, P.J.; Jones, D.K.
\newblock Diffusion-tensor {MRI}: Theory, experimental design and data analysis---A technical review.
\newblock {\em NMR Biomed.} {\bf 2002}, {\em 15},~456--467.
https://doi.org/10.1002nbm.783.

\bibitem[Pajevic and Basser(2003)]{doi:10.1016/s1090-7807(02)00178-7}
Pajevic, S.; Basser, P.J.
\newblock Parametric and non-parametric statistical analysis of {DT-MRI} data.
\newblock {\em J. Magn. Reson.} {\bf 2003}, {\em 161},~1--14.
https://doi.org/10.1016/s1090-7807(02)00178-7.

\bibitem[Basser and Pajevic(2003)]{doi:10.1109/TMI.2003.815059}
Basser, P.J.; Pajevic, S.
\newblock A normal distribution for tensor-valued random variables:
  Applications to diffusion tensor {MRI}.
\newblock {\em IEEE Trans. Med. Imaging} {\bf 2003}, {\em 22},~785--794.
\newblock
  https://doi.org/10.1109/TMI.2003.815059.

\bibitem[Gasbarra {et~al.}(2017)Gasbarra, Pajevic, and
  Basser]{doi:10.1137/16M1098693}
Gasbarra, D.; Pajevic, S.; Basser, P.J.
\newblock Eigenvalues of random matrices with isotropic {G}aussian noise and
  the design of diffusion tensor imaging experiments.
\newblock {\em SIAM J. Imaging Sci.} {\bf 2017}, {\em 10},~1511--1548.
\newblock
  https://doi.org/10.1137/16M1098693.

\bibitem[Alexander {et~al.}(2001)Alexander, Pierpaoli, Basser, and
  Gee]{doi:10.1109/42.963816}
Alexander, D.C.; Pierpaoli, C.; Basser, P.J.; Gee, J.C.
\newblock Spatial transformations of diffusion tensor magnetic resonance
  images.
\newblock {\em IEEE Trans. Med. Imaging} {\bf 2001}, {\em 20},~1131--1139.
https://doi.org/10.1109/42.963816.

\bibitem[Schwartzman {et~al.}(2008)Schwartzman, Mascarenhas, and
  Taylor]{MR2485016}
Schwartzman, A.; Mascarenhas, W.F.; Taylor, J.E.
\newblock Inference for eigenvalues and eigenvectors of {G}aussian symmetric
  matrices.
\newblock {\em Ann. Statist.} {\bf 2008}, {\em 36},~2886--2919.

\bibitem[Mallows(1961)]{MR131312}
Mallows, C.L.
\newblock Latent vectors of random symmetric matrices.
\newblock {\em Biometrika} {\bf 1961}, {\em 48},~133--149.

\bibitem[Hu and White(1997)]{doi:10.1016/S1384-1076(97)00022-5}
Hu, W.; White, M.
\newblock A {CMB} polarization primer.
\newblock {\em New Astron.} {\bf 1997}, {\em 2},~323--344.
\newblock
 https://doi.org/10.1016/S1384-1076(97)00022-5.

\bibitem[Vafaei Sadr and Movahed(2021)]{doi:10.1093/mnras/stab368}
Vafaei Sadr, A.; Movahed, S.M.S.
\newblock Clustering of local extrema in {P}lanck {CMB} maps.
\newblock {\em MNRAS} {\bf 2021}, {\em 503},~815--829.
\newblock
https://doi.org/10.1093/mnras/stab368.

\bibitem[Gallaugher and McNicholas(2018)]{doi:10.1016/j.patcog.2018.02.025}
Gallaugher, M.P.B.; McNicholas, P.D.
\newblock Finite mixtures of skewed matrix variate distributions.
\newblock {\em Pattern Recognit.} {\bf 2018}, {\em 80},~83--93.
\newblock
https://doi.org/10.1016/j.patcog.2018.02.025.

\bibitem[Ouimet(2022)]{Ouimet_2022_Student_JCA}
Ouimet, F.
\newblock Refined normal approximations for the {S}tudent distribution.
\newblock {\em J. Classical Anal.} {\bf 2022}, {\em 20},~23--33.

\bibitem[Shafiei and Saberali(2015)]{doi:10.1109/LCOMM.2015.2442576}
Shafiei, A.; Saberali, S.M.
\newblock A simple asymptotic bound on the error of the ordinary normal
  approximation to the {S}tudent's $t$-distribution.
\newblock {\em IEEE Commun. Lett.} {\bf 2015}, {\em 19},~1295--1298.

\bibitem[Govindarajulu(1965)]{MR207011}
Govindarajulu, Z.
\newblock Normal approximations to the classical discrete distributions.
\newblock {\em Sankhy\={a} Ser. A} {\bf 1965}, {\em 27},~143--172.

\bibitem[Esseen(1945)]{MR14626}
Esseen, C.G.
\newblock Fourier analysis of distribution functions. {A} mathematical study of
  the {L}aplace-{G}aussian law.
\newblock {\em Acta Math.} {\bf 1945}, {\em 77},~1--125.

\bibitem[Cressie(1978)]{MR538319}
Cressie, N.
\newblock A finely tuned continuity correction.
\newblock {\em Ann. Inst. Statist. Math.} {\bf 1978}, {\em 30},~435--442.

\bibitem[Gaunt(2014)]{MR3194737}
Gaunt, R.E.
\newblock Variance-gamma approximation via {S}tein's method.
\newblock {\em Electron. J. Probab.} {\bf 2014}, {\em 19},~1--33.

\bibitem[Gaunt(2021)]{MR4291370}
Gaunt, R.E.
\newblock New error bounds for {L}aplace approximation {\it via} {S}tein's
  method.
\newblock {\em ESAIM Probab. Stat.} {\bf 2021}, {\em 25},~325--345.

\bibitem[Gaunt(2020)]{MR4064309}
Gaunt, R.E.
\newblock Wasserstein and {K}olmogorov error bounds for variance-gamma
  approximation via {S}tein's method {I}.
\newblock {\em J. Theoret. Probab.} {\bf 2020}, {\em 33},~465--505.

\bibitem[{R Core Team}(2020)]{Rsoftware}
{R Core Team}.
\newblock {\em R: A Language and Environment for Statistical Computing};
\newblock R Foundation for Statistical Computing: Vienna, Austria,  2020.

\bibitem[Hotelling(1931)]{doi:10.1214/aoms/1177732979}
Hotelling, H.
 The generalization of {S}tudent's ratio.
 {\em Ann. Math. Statist.} {\bf 1931}, {\em 2},~360--378.
  https://doi.org/10.1214/aoms/1177732979.

\bibitem[Abramowitz and Stegun(1964)]{MR0167642}
Abramowitz, M.; Stegun, I.A.
\newblock {\em Handbook of {M}athematical {F}unctions with {F}ormulas,
  {G}raphs, and {M}athematical {T}ables}; {National Bureau of
  Standards Applied Mathematics Series}, For sale by the Superintendent of
  Documents, U.S. Government Printing Office: Washington, DC, USA, 1964; Volume~55, pp.
  xiv+1046.

\bibitem[Carter(2002)]{MR1922539}
Carter, A.V.
\newblock Deficiency distance between multinomial and multivariate normal
  experiments.
\newblock {\em Ann. Statist.} {\bf 2002}, {\em 30},~708--730.

\bibitem[de~Waal and Nel(1973)]{MR347003}
de~Waal, D.J.; Nel, D.G.
\newblock On some expectations with respect to {W}ishart matrices.
\newblock {\em South African Statist. J.} {\bf 1973}, {\em 7},~61--67.

\bibitem[Letac and Massam(2004)]{MR2066255}
Letac, G.; Massam, H.
\newblock All invariant moments of the {W}ishart distribution.
\newblock {\em Scand. J. Statist.} {\bf 2004}, {\em 31},~295--318.

\bibitem[Haff(1979)]{MR556910}
Haff, L.R.
\newblock An identity for the {W}ishart distribution with applications.
\newblock {\em J. Multivar. Anal.} {\bf 1979}, {\em 9},~531--544.

\bibitem[von Rosen(1988)]{MR968156}
von Rosen, D.
\newblock Moments for the inverted {W}ishart distribution.
\newblock {\em Scand. J. Statist.} {\bf 1988}, {\em 15},~97--109.

\end{thebibliography}

\end{document}